\documentclass{amsart}
\usepackage{graphicx}
\usepackage{latexsym}
\usepackage{hyperref}
\usepackage{amsfonts}
\usepackage[all]{xy}
\usepackage{amssymb, mathrsfs, amsfonts, amsmath}
\usepackage{color}
\usepackage{appendix}

\usepackage{enumerate}

\usepackage{eulervm}
\usepackage{geometry}

\newenvironment{\thesection}{\section}{Appendix}

\newtheorem{theorem}{Theorem}[section]

\newtheorem*{acknowledgement}{Acknowledgement}

\newtheorem{corollary}[theorem]{Corollary}

\newtheorem{definition}[theorem]{Definition}
\newtheorem{example}[theorem]{Example}

\newtheorem{remark}[theorem]{Remark}

\newcommand{\R}{\mathbb{R}}

\newcommand{\Sym}{\mathrm{Sym}}
\newcommand{\metric}{\langle \, , \, \rangle}

\newcommand{\disp}{\displaystyle}

\newcommand{\ra}{\rightarrow}

\newcommand{\eps}{\varepsilon}

\newcommand{\II}{\mathrm{II}}

\newcommand{\di}{\mathrm{d}}
\newcommand{\weak}{\mathrm{w}}

\newcommand{\EE}{\mathscr{E}}
\newcommand{\HH}{\mathscr{H}}

\newcommand{\Ricc}{\mathrm{Ric}}
\newcommand{\Sect}{\mathrm{Sect}}
\newcommand{\cut}{\mathrm{cut}}
\newcommand{\vol}{\mathrm{vol}}
\newcommand{\lip}{\mathrm{Lip}}
\newcommand{\cal}{\mathcal}
\newcommand{\rad}{\mathrm{rad}}
\newcommand{\capac}{\mathrm{cap}}
\newcommand{\diver}{\mathrm{div}}
\newcommand{\loc}{\mathrm{loc}}

\DeclareMathOperator{\tr}{\mathrm{Tr}}

\begin{document}

\title[Maximum principles and AK-duality]{Maximum principles at infinity and the Ahlfors-Khas'minskii duality: an overview}  
{\author{Luciano Mari}
\address{Dipartimento di Matematica\\Scuola Normale Superiore\\56126 Pisa-Italy}
\email{luciano.mari@sns.it, mari@mat.ufc.br} \thanks{The first author is supported by the grants SNS17\_B\_MARI and SNS\_RB\_MARI of the Scuola Normale Superiore.}
{\author{Leandro F. Pessoa}
\address{Departamento de Matem\'{a}tica\\Universidade Federal do Piau\'{i}-UFPI\\
64049-550, Teresina-Brazil} \email{leandropessoa@ufpi.edu.br} \thanks{The second author was partially supported by CNPq-Brazil.}}

\maketitle

\begin{abstract}
This note is meant to introduce the reader to a duality principle for nonlinear equations recently discovered in \cite{valtorta, marivaltorta, maripessoa}. Motivations come from the desire to give a unifying potential-theoretic framework for various maximum principles at infinity appearing in the literature (Ekeland, Omori-Yau, Pigola-Rigoli-Setti), as well as to describe their interplay with properties coming from stochastic analysis on manifolds. The duality involves an appropriate version of these principles formulated for viscosity subsolutions of fully nonlinear inequalities, called the Ahlfors property, and the existence of suitable exhaustion functions called Khas'minskii potentials. Applications, also involving the geometry of submanifolds, will be discussed in the last sections. We conclude by investigating the stability of these maximum principles when we remove polar sets.
\end{abstract}

\tableofcontents

\section{Prelude: maximum principles at infinity} 
Maximum principles at infinity are a powerful tool to investigate problems in Geometry. They arose from the desire to generalize the statement that, on a compact Riemannian manifold $(X, \metric)$ of dimension $m \ge 2$, every function $u \in C^2(X)$ attains a maximum point $x_0$ and 
\begin{equation}\label{finite_max}
(i) \ : \  \ u(x_0) = \sup_X u, \qquad (ii) \ : \ \ |\nabla u(x_0)| = 0, \qquad (iii) \ : \ \  \nabla^2 u(x_0) \le 0, 
\end{equation}
where $\nabla^2 u$ is the Riemannian Hessian and the last relation is meant in the sense of quadratic forms. If $M$ is noncompact and given $u \in C^2(X)$ bounded from above, although one cannot ensure the existence of a maximum point, there could still exist a sequence $\{x_k\} \subset X$ such that some of the relations in \eqref{finite_max} hold in a limit sense as $k \ra \infty$. Informally speaking, when this happens we could think that $X$ is ``not too far from being compact". The first example of maximum principle at infinity is the famous Ekeland's principle, \cite{ekeland_2}, that can be stated as follows:
\begin{definition}\label{def_ekeland_classico}
A metric space $(X,\di)$ satisfies the \emph{Ekeland maximum principle} if, for each $u$ upper semincontinuous (USC) on $X$ and bounded from above, there exists a sequence $\{x_k\} \subset X$ with the following properties:
$$
u(x_k) > \sup_X u - \frac{1}{k}, \qquad u(y) \le u(x_k) + \frac{1}{k}\di(x_k,y) \ \text{ for each } \, y \in X.
$$
\end{definition}
The full statement of Ekeland's principle contains, indeed, a further property that is crucial in applications, that is, the possibility to create one such $\{x_k\}$ suitably close to a given maximizing sequence $\{\bar x_k\}$. We will briefly touch on it later. By works of I. Ekeland, J.D. Weston and F. Sullivan, cf. \cite{ekeland_2, weston, sullivan}, the validity of Ekeland's principle in the form given above is in fact \emph{equivalent} to $X$ being a complete metric space. Therefore, in the smooth setting, the (geodesic) completeness of a manifold $X$ enables to find $\{x_k\}$ approximating both $(i)$ and $(ii)$ in \eqref{def_ekeland_classico}. However, condition $(iii)$ requires further restrictions on the geometry of $X$, first investigated by H. Omori \cite{omori} and S.T. Yau \cite{chengyau, yau} (the second with S.Y. Cheng). They introduced the following two principles, respectively, in the Hessian case \cite{omori} and in the Laplacian case \cite{chengyau, yau}:
\begin{definition}\label{def_omoriyau}
Let $(X, \metric)$ be a Riemannian manifold. We say that $X$ satisfies the \emph{strong Hessian (respectively, Laplacian) maximum principle} if, for each $u \in C^2(X)$ bounded from above, there exists a sequence $\{x_k\} \subset X$ with the following properties:
\begin{equation}\label{strong_HessianLaplacian}
\begin{array}{ll}
(\text{Hessian}) & \quad u(x_k) > \sup_X u - k^{-1}, \qquad |\nabla u(x_k)| < k^{-1}, \qquad \nabla \di u(x_k) \le k^{-1} \metric ;\\[0.2cm]
(\text{Laplacian}) & \quad u(x_k) > \sup_X u - k^{-1}, \qquad |\nabla u(x_k)| < k^{-1}, \qquad \Delta u(x_k) \le k^{-1}.
\end{array}
\end{equation}
\end{definition}
Here, the word ``strong" refers to the presence of the condition on the gradient. Historically, these principles are called the \emph{Omori-Yau maximum principles}, and proved to be remarkably effective in a wealth of different geometric problems. Among them, we stress the striking proofs of the generalized Schwarz Lemma for maps between Kahler manifolds in \cite{yau_schwarz}, and of the Bernstein theorem for maximal hypersurfaces in Minkovski space in \cite{chengyau_minkovski}. \par

Geometric conditions to guarantee the Omori-Yau principles are often expressed in terms of growths of the curvatures of $X$ with respect to the distance $\varrho(x)$ from a fixed origin $o \in X$, but are not necessarily depending on them. The most general  known condition guaranteeing the Omori-Yau principles is given by \cite{prsmemoirs} (cf. also improvements in \cite{borbely, bessalimapessoa}): the principle holds whenever $X$ supports a function $w$ with the following properties\footnote{Here, as usual, if we write ``$w(x) \ra +\infty$ as $x$ diverges" we mean that the sublevels of $w$ have compact closure in $X$, that is, that $w$ is an exhaustion.}:
\begin{equation}\label{strongKhasminskii}
\begin{array}{l}
0 < w \in C^2(X\backslash K) \ \text{ for some compact $K$}, \qquad w \ra + \infty \ \text{ as } \, x \ \text{ diverges}, \\[0.2cm]
|\nabla w| \le G(w), \qquad \text{and} \quad \left\{ \begin{array}{ll} 
\nabla^2 w \le G(w) \metric & \quad \text{for Omori's principle}, \\[0.1cm]
\Delta w \le G(w) & \quad \text{for Yau's principle},
\end{array}\right.
\end{array}
\end{equation} 
for some $G$ satisfying
\begin{equation}\label{ipo_G_SMP}
0< G \in C^1(\R^+), \qquad G' \ge 0, \qquad \int^{+\infty} \frac{\di s}{G(s)} = +\infty. 
\end{equation}
For instance, $G(t) = (1+t)$ gives the sharp polynomial threshold. The criterion is effective, since the function in \eqref{strongKhasminskii} can be explicitly found in a number of geometrically relevant applications: for instance, if the radial sectional curvature\footnote{The radial sectional curvature is the sectional curvature restricted to $2$-planes containing $\nabla \varrho$. Inequality $\Sect_\rad \ge - G^2(\varrho)$ means that $\Sect(\pi_x) \ge - G^2\big(\varrho(x)\big)$ for each $x \not\in \{o\} \cup \cut(o)$ and $\pi_x \le T_x X$ $2$-plane containing $\nabla \rho$.} (respectively, Ricci curvature) is bounded from below as follows:
\begin{equation}\label{bounds_omoriyau}
\begin{array}{ll}
\Sect_\rad \ge - G^2(\varrho) \qquad \text{on } \, X, & \quad \text{for Omori's principle,} \\[0.2cm]
\Ricc(\nabla \varrho, \nabla \varrho) \ge - G^2(\varrho) \qquad \text{on } \, X \backslash \{ \{o\} \cup \cut(o)\} & \quad \text{for Yau's principle,}
\end{array}
\end{equation}
then one can choose $w(x) = \log(1+\varrho(x))$ in \eqref{strongKhasminskii} to deduce the validity of the strong Hessian, respectively, strong Laplacian principle (technically, $\varrho$ is not $C^2$, but one can overcome the problem by using Calabi's trick, see \cite{prsmemoirs}). Note that \eqref{bounds_omoriyau} includes the case when $G$ is constant, considered in \cite{omori, yau}.
However, the existence of $w$ could be granted even without bounds like \eqref{bounds_omoriyau}: for  instance, in the strong Laplacian case, this happens if $X$ is properly immersed with bounded mean curvature in $\R^m$ (or in a Cartan-Hadamard ambient space with bounded sectional curvature), or if $X$ is a Ricci soliton, see \cite{aliasmastroliarigoli}. The function $w$ in \eqref{strongKhasminskii} is an example of what we will call a \emph{Khas'minskii potential}. The reason for the name will be apparent in a moment.\par

\subsection{An example: immersions into cones}

There is, by now, a wealth of applications of the Omori-Yau principles in geometry, see for instance  \cite{aliasmastroliarigoli}. Here, we illustrate how the principles can be effectively used in geometry by means of the following example in \cite{maririgoli}, that is related to the pioneering paper by Omori \cite{omori}. Hereafter, a non-degenerate cone $\mathcal{C}_{o,v,\eps}$ of center $o \in \R^n$, axis $v \in \mathbb{S}^{n-1}$ and width $\eps \in (0,\pi/2)$ is the set of points $x \in \R^n$ such that
$$
\langle \frac{x-o}{|x-o|}, v \rangle \ge \cos \eps. 
$$

\begin{theorem}\cite[Cor. 1.18]{maririgoli}
Let $\varphi : X^m \ra \R^{2m-1}$ be an isometric immersion. Denoting with $\rho$ the distance from a fixed origin $o$, assume that the sectional curvature of $X$ satisfies
\begin{equation}\label{eq_omori}
- C(1+\rho^2) \le \Sect \le 0, \qquad \text{on } \, M,
\end{equation}
for some constant $C>0$. Then, $\varphi(X)$ cannot be contained into any non-degenerate cone of $\R^{2m-1}$.
\end{theorem}

\begin{proof}[{\bf Sketch of the proof:}]
Suppose, by contradiction, that $\varphi(X) \subset \mathcal{C}_{o,v,\eps}$ for some $o,v,\eps$. Without loss of generality, we can assume that $o$ is the origin of $\R^{2m-1}$. The first step is to construct a function that encodes the geometry of the problem at hand: following \cite{omori}, we fix $x_0 \in X\backslash\{o\}$ and $a \in (0,\cos\eps)$, we set $T = \langle \varphi(x_0), v \rangle$ and we define 
$$
u(x) = \sqrt{T^2 + a^2|\varphi(x)|^2} - \langle \varphi(x),v\rangle. 
$$
By construction, it is easy to show that $u<T$ on $X$ and that the non-empty upper level set $\{u>0\}$ (that contains $x_0$) has bounded image $\varphi\big(\{u>0\}\big)$. In view of \eqref{eq_omori} and the discussion above, the strong Hessian principle holds and can be applied to $u$. However, for $x \in \{u>0\}$ and unit vector $W \in T_xX$, a computation shows that
\vspace{0.2cm}
\begin{eqnarray*}
\disp \nabla^2 u(W,W) &=& \disp \frac{a^2( 1 + \langle \II(W,W), \varphi \rangle}{\sqrt{T^2+a^2|\varphi|^2}} - \langle\II(W,W),v\rangle - \frac{a^4 \langle W, \varphi \rangle^2}{(T^2+a^2|\varphi|^2)^{3/2}} \\[0.2cm]
& \ge & \disp \frac{a^2( 1 - |\II(W,W)||\varphi|)}{\sqrt{T^2+a^2|\varphi|^2}} - |\II(W,W)| - \frac{a^4 |\varphi|^2}{(T^2+a^2|\varphi|^2)^{3/2}} \\[0.2cm]
& \ge & \disp \frac{a^2T^2}{(T^2+a^2|\varphi|^2)^{3/2}} - |\II(W,W)| g(x),
\end{eqnarray*}
for some continuous function $g$ on $\{ u>0\}$. Since the codimension of $\varphi$ is strictly less than $m$ and the sectional curvature is non-positive, a useful algebraic lemma due to Otsuki \cite{docarmo, prsmemoirs} guarantees the existence of $W \in T_xX$ such that $\II(W,W)=0$. Having fixed such $W$, and taking into account that $\varphi(\{u>0\})$ is bounded, there exists a uniform constant $c>0$ such that 
$$
\sup_{Z \in T_xX, \ |Z|=1} \nabla^2 u(Z,Z) \ge \nabla^2 u(W,W) \ge \frac{a^2T^2}{(T^2+ a^2|\varphi|^2)^{3/2}} \ge c>0 \qquad \text{on } \, \{u>0\}.
$$ 
However, evaluating the above inequality on a sequence $\{x_k\}$ realizing the strong Hessian principle we obtain a contradiction.
\end{proof}

\begin{remark}
\emph{In its full strength, Ekeland's principle also guarantees that the sequence $\{x_k\}$ satisfying \eqref{def_ekeland_classico} can be chosen to be close to a given maximizing sequence $\{\bar x_k\}$ with explicit bounds, a fact that is very useful in applications to functional analysis and PDEs. On the other hand, to present no systematic investigation of an analogous property was performed for the Omori-Yau principles. A notable exception, motivated by a geometrical problem involving convex hulls of isometric immersions, appeared in \cite{fontenelexavier}: there, the authors proved that if $X$ is complete and $\Sect \ge - c$ for some $c \in \R^+$ (resp. $\Ricc \ge -c$), any maximizing sequence $\{\bar x_k\}$ has a \emph{good shadow}, that is, a sequence $\{x_k\}$ satisfying Omori (resp. Yau) principle  and also $\di(x_k,\bar x_k) \ra 0$ as $k \ra \infty$.  
}
\end{remark}

The properties in \eqref{strong_HessianLaplacian} can be rephrased as follows, say in the Hessian case: denoting with $\lambda_1(A) \le \lambda_2(A) \le \ldots \le \lambda_m(A)$ the eigenvalues of a symmetric matrix $A$, $X$ has the strong Hessian principle if either one of the following properties holds:
\begin{itemize}
\item[(i)] it is not possible to find a function $u \in C^2(X)$ bounded from above and such that, on some non-empty upper level-set $\{u> \gamma\}$, 
\begin{equation}\label{eq_stronghessian}
\max\Big\{ |\nabla u| -1, \lambda_m(\nabla^2 u) - 1\Big\} \ge 0.
\end{equation}
\item[(ii)] for every open set $U \subset M$ and every $u \in C^2(U)\cap C(\overline{U})$ bounded from above and solving the differential inequality
$$
\max\Big\{ |\nabla u| -1, \lambda_m(\nabla^2 u) - 1\Big\} \ge 0,
$$
it holds $\sup_U u = \sup_{\partial U} u$.
\end{itemize}
By rescaling, the constant $1$ can be replaced by any fixed positive number. It is evident that the strong Hessian principle is equivalent to (i), while (i)$\Leftrightarrow$(2) can easily be proved by contradiction: if we assume that (i) fails for some $u$, consider as $U$ the upper level set where \eqref{eq_stronghessian} holds and contradict (ii), while if (ii) fails for some $u$, then take any upper level set at height $\gamma \in (\sup_{\partial U} u, \sup_U u)$ to contradict (i).\par
Conditions (i) and (ii) are invariant by translations $u \mapsto u + \mathrm{const}$. For future use it is important to describe another characterization, not translation invariant, where the constant $1$ in (ii) is replaced by a pair of functions 
\begin{equation}\label{fxi}\tag{$f\xi$}
\left\{ \begin{array}{l}
f \in C(\R), \quad f(0)=0, \quad f>0 \ \ \text{on } \, \R^+, \quad f \ \text{ is odd and strictly increasing;} \\[0.2cm]
\xi \in C(\R), \quad \xi(0)=0, \quad \xi<0, \ \ \text{on } \, \R^+, \quad \xi \ \text{ is odd and strictly  decreasing.} 
\end{array}
\right.
\end{equation}
Namely, (i) and (ii) are also equivalent to
\begin{itemize}
\item[($\mathscr{A}$)] for some (\emph{equivalently}, any) pair $(f,\xi)$ satisfying \eqref{fxi}, the following holds: for every open set $U \subset M$ and every $u \in C^2(U)\cap C(\overline{U})$ bounded from above and solving the differential inequality
\begin{equation}\label{eq_stronghessian_uguale}
\max\Big\{ |\nabla u| - \xi(-u), \ \lambda_m(\nabla^2 u) - f(u)\Big\} \ge 0 \qquad \text{on } \, \big\{ x \ : \ u(x)>0\big\} \neq \emptyset,  
\end{equation}
it holds $\sup_U u = \sup_{\partial U} u$.
\end{itemize}
The choice of the upper level-set $\{u>0\}$ is related to the vanishing of $f,\xi$ in \eqref{fxi} at zero and is, of course, just a matter of convenience. The proof of (ii)$\Leftrightarrow$($\mathscr{A}$) proceeds by translation and rescaling arguments, and localizing on suitable upper level sets of $u$, and is given in detail in \cite[Prop. 5.1]{maripessoa}. From the technical point of view, the dependence of ($\mathscr{A}$) on $f,\xi$ just in terms of the mild properties in \eqref{fxi}, and especially the possibility to check ($\mathscr{A}$) in terms of a \emph{single} pair $f,\xi$ satisfying \eqref{fxi}, is useful in applications.

Characterization ($\mathscr{A}$) is in the form of a maximum principle on sets with boundary, for subsolutions of the fully nonlinear inequality
$$
\mathscr{F}(x,u(x), \nabla u(x), \nabla^2 u(x) \big) \doteq \max\Big\{ |\nabla u| - \xi(-u), \ \lambda_m(\nabla^2 u) - f(u)\Big\} \ge 0.
$$
This point of view relates the principles to another property that can be seen as a replacement of the compactness of $X$, the \emph{parabolicity} of $X$. We recall that
\begin{definition}
A manifold $X$ is said to be \emph{parabolic} if each solution of $\Delta u \ge 0$ on $X$ that is bounded from above is constant.
\end{definition}

Indeed, L.V. Ahlfors (see \cite[Thm. 6C]{ahlforssario}) observed that $X$ is parabolic if and only if property ($\mathscr{A}$) holds with \eqref{eq_stronghessian_uguale} replaced by $\Delta u \ge 0$. The problem of deciding whether a manifold is parabolic or not is classical, and there is by now a well established theory, see \cite{grigoryan} for a thorough account. For surfaces, the theory arose in connection to the \emph{type problem} for Riemann surfaces, cf. \cite{ahlforssario}, and arguments involving parabolicity are still crucial in establishing a number of powerful, recent results in modern minimal surface theory (see for instance \cite{mazet, meeksrosenberg, cokumero} for beautiful examples).\par 

The tight relation between parabolicity and potential theory suggests that it might be possible to treat both Ekeland and Omori-Yau principles as well in terms of a fully nonlinear potential theory. In recent years, there has been an increasing interest in the theory of fully-nonlinear PDEs on manifolds, and especially R. Harvey and B. Lawson dedicated a series of papers \cite{HL_primo, HL_dir, HL_existence, HL_plurisub, HL_removable, HL_SMP} to develop a robust geometric approach for fully nonlinear PDEs well suited to do potential theory for those equations. Their work fits perfectly to the kind of problems considered in the present paper, and will be introduced later. Our major concern in the recent \cite{marivaltorta, maripessoa, mmr_berestycki} is to put the above principles, as well as other properties to be discussed in a moment, into a unified framework where new relations, in particular an underlying duality, could emerge between them.

\subsection{Parabolicity, capacity and Evans potentials}

There are a number of equivalent conditions characterizing the parabolicity of $X$, see \cite[Thm. 5.1]{grigoryan} and \cite{pigolasetti_ensaio}, and we now focus on two of them.\par
The first one describes parabolic manifolds as those for which the $2$-capacity of every compact $K$ vanishes. We recall that, for fixed $q \in (1,\infty)$, the $q$-capacity of a condenser $(K, \Omega)$ with $K \subset \Omega \subset X$, $K$ compact, $\Omega$ open, is the following quantity:
\begin{equation}\label{def_kcapac}
\capac_q(K,\Omega) = \inf \left\{  \int_\Omega |\nabla \phi|^q, \ \ : \ \  \phi \in \lip_c(\Omega), \ \ \phi \ge 1 \ \text{on $K$}\right\}.
\end{equation}
If $K$ and $\Omega$ have Lipschitz boundary, the infimum is realized by the unique solution $u$ of the $q$-Laplace equation
\begin{equation}\label{def_kcapacitor}
\left\{ \begin{array}{l}
\Delta_q u \doteq \diver \left( |\nabla u|^{q-2} \nabla u\right) = 0 \qquad \text{on } \, \Omega \backslash K, \\[0.2cm]
u=1 \quad \text{on } \, K, \qquad u=0 \quad \text{on } \, \partial \Omega,
\end{array}\right.
\end{equation}
called the $q$-capacitor of $(K, \Omega)$. A manifold is called $q$-parabolic if $\capac_q(K,X) = 0$ for some (equivalently, every) compact set $K$, see \cite{holopainen, troyanov2}. By extending Ahlfors result for the Laplace-Beltrami operator \cite{prsnonlinear, prs_milan}, 
\begin{equation}\label{ahlfors_kparab}
\text{$X$ is $q$-parabolic} \qquad \Longleftrightarrow \qquad \left\{ \begin{array}{c}
\forall \, U \subset X \ \text{ open, } \ \forall u \in C^2(U)\cap C(\overline{U}),\\[0.2cm]
\text{bounded above and solving $\Delta_q u \ge 0$,} \\[0.2cm]
\text{it holds } \quad \sup_U u = \sup_{\partial U} u,
\end{array}\right\}
\end{equation}
and both are equivalent to the constancy of solutions of $\Delta_q u \ge 0$ that are bounded from above.\par 
There is a further, quite useful characterization of parabolicity (in the linear setting $q=2$), expressed in terms of suitable exhaustion functions and studied by Z. Kuramochi and M. Nakai \cite{kuramochi, nakai, sarionakai}, with previous contribution by R.Z. Khas'minskii \cite{khasminski}. In \cite{kuramochi, nakai, sarionakai}, the authors proved that $X$ is parabolic if and only if, for each compact set $K$ with smooth boundary, there exists a function $w$ solving
\begin{equation}\label{eq_evans}
\begin{array}{l}
w \in C^\infty(X \backslash K), \qquad w > 0 \ \ \text{ on } \, X \backslash K, \qquad w=0 \ \ \text{ on } \, \partial K, \\[0.2cm]
w(x) \ra +\infty \ \ \text{ as } \, x \text{ diverges,} \qquad \Delta w = 0 \quad \text{on } \, X \backslash K. 
\end{array}
\end{equation}
Such a $w$ is named an \emph{Evans potential} on $X \backslash K$. Evans potentials proved to be useful in investigating the topology of $X$ by means of the beautiful Li-Tam-Wang's theory of harmonic functions, cf. \cite{litam, sungtamwang} (cf. also \cite{li, prs} for comprehensive accounts), and it is therefore of interest to see whether other Liouville type properties could be characterized in terms of Evans potentials. One quickly realizes that, for this to hold, the function $\mathscr{F}$ replacing $\Delta$ in \eqref{eq_evans} must have a very specific form, and in fact, to present, the Laplace-Beltrami is the only operator for which solutions $w$ in \eqref{eq_evans} \emph{with the equality sign} have been constructed. The reason is that the proof in \cite{kuramochi, nakai, sarionakai}, see also \cite{valtorta}, strongly uses the characterization of parabolicity in terms of the $2$-capacity and the linearity of the Laplace-Beltrami operator. Even for $q \neq 2$, the equivalence of $q$-parabolicity with the existence of $q$-harmonic Evans potentials is still an open problem, although results for more general operators on rotationally symmetric manifolds (cf. the last section in \cite{marivaltorta}) indicate that it is likely to hold.\par 
Quite differently, if we just require that $w$ be a \emph{supersolution}, that is, $\Delta_q w \le 0$, things are much more flexible and are still worth interest, as we shall see later. In \cite{valtorta} the author proved that the $k$-parabolicity is equivalent to the existence of $w$ satisfying \eqref{eq_evans} with the last condition weakened to $\Delta_q w \le 0$. Although part of the proof uses $q$-capacities and is therefore very specific to the $q$-Laplacian, the underlying principle is general: the construction of $w$ proceeds by ``stacking" solutions of suitable obstacle problems, an idea that will be described later in a more general framework.\par
It is natural to ask what is the picture for $p = + \infty$, that is, setting,
$$
\capac_\infty(K,\Omega) = \inf \Big\{  \|\nabla \phi\|_\infty, \ \ : \ \ \phi \in \lip_c(\Omega), \ \ \phi \ge 1 \ \text{on $K$}\Big\},
$$
to study $\infty$-parabolic manifolds, defined as those for which $\capac_\infty(K,X)=0$ for  every compact $K$. The problem has recently been addressed in \cite{pigolasetti_ensaio}, where the authors proved that 
$$
\text{$X$ is $\infty$-parabolic} \qquad \Longleftrightarrow \qquad \text{$X$ is (geodesically) complete}
$$
and thus, a-posteriori, $\infty$-parabolicity is equivalent to Ekeland's principle. Below, we shall complement these characterizations as applications of our main duality principle. To do so, we exploit  the existence of $\infty$-capacitors for $(K, \Omega)$, that is, \emph{suitable} minimizers (absolutely minimizing Lipschitz extensions, cf. \cite{crandall_visit, juutinen}) for $\capac_\infty(K,\Omega)$. Their existence was first considered by G. Aronsson \cite{aronsson} and proved by R. Jensen \cite{jensen} when $X=\R^m$: the $\infty$-capacitor turns out to be the (unique) solution of  
\begin{equation}\label{def_inftycapacitor}
\left\{ \begin{array}{l}
\Delta_\infty u \doteq \nabla^2 u(\nabla u, \nabla u) = 0 \quad \text{on } \, \Omega \backslash K, \\[0.2cm]
u=1 \quad \text{on } \, K, \qquad u=0 \quad \text{on } \, \partial \Omega,
\end{array}\right.
\end{equation}
where the equation is meant in the viscosity sense (see below). The operator $\Delta_\infty$, called the infinity Laplacian, has recently attracted a lot of attention because of its appearance in Analysis, Game Theory and Physics, and its investigation turns out to be challenging because of its high degeneracy (see \cite{jensen, crandall_visit, CEG}).

\subsection{Link with stochastic processes}
Before introducing the duality, we mention some other important function-theoretic properties of $X$, coming from stochastic analysis, that fit well with our setting and give further geometric motivation. We start recalling that parabolicity can be further characterized in terms of the Brownian motion on $X$. Briefly, on each Riemannian manifold $X$ one can construct the heat kernel $p(x,y,t)$ (cf. \cite{dodziuk}), and consequently a stochastic process $\mathscr{B}_t$ whose infinitesimal generator is $\Delta$, called the Brownian motion, characterized by the identity
\begin{equation}\label{prob_brownian}
\mathbb{P}\left( \mathscr{B}_t \in \Omega \ : \ \mathscr{B}_0 = x\right) = \int_\Omega p(x,y,t) \di y,
\end{equation}
for every open subset $\Omega \subset X$ (see \cite{bar} for a beautiful, self-contained introduction). 
\begin{definition}
Let $\mathscr{B}_t$ be the Brownian motion on $X$.
\begin{itemize}
\item $\mathscr{B}_t$ on $X$ is called \emph{recurrent} if, almost surely, its trajectories visit every fixed compact set $K \subset X$ infinitely many times.
\item $\mathscr{B}_t$ on $X$ is called \emph{non-explosive} if, almost surely, its trajectories do not escape to infinity in finite time, that is, \eqref{prob_brownian} with $\Omega = X$ is identically $1$ for some (equivalently, every) $(x,t) \in X \times \R^+$. 
\end{itemize}
\end{definition}
A manifold whose Brownian motion is non-explosive is called \emph{stochastically complete}. The recurrency of $\mathscr{B}_t$ is equivalent to the parabolicity of $X$, cf. \cite{grigoryan}, and can therefore be characterized in terms of a property of type ($\mathscr{A}$) above. Similarly, by \cite[Thm. 6.2]{grigoryan}, $X$ is stochastically complete if and only if there exists no open subset $U \subset X$ supporting a solution $u$ of 
\begin{equation}\label{ahlfors_stoca}
\left\{\begin{array}{l}
\disp \Delta u \ge \lambda u \qquad \text{on } \, U, \quad \text{for some} \, \lambda \in \R^+, \\[0.2cm]
\disp u>0 \quad \text{on } \, U, \qquad u =0 \quad \text{on } \, \partial U. 
\end{array}\right.
\end{equation}
for some (equivalently, any) fixed $\lambda  \in \R^+$. The last property is ($\mathscr{A}$) provided that \eqref{eq_stronghessian} is replaced by $\Delta u -\lambda u \ge 0$, that is, choosing $f(r) = \lambda r$ and removing the gradient condition. Again by translation, rescaling and localizing arguments, the function $\lambda r$ can be replaced by any $f$ satisfying \eqref{fxi}.\par
Around 15 years ago, new interest arose around the notion of stochastic completeness, after the observation in \cite{prs_proceeding, prsmemoirs} that the property is \emph{equivalent} to a relaxed form of the strong Laplacian principle, called the \emph{weak (Laplacian) maximum principle}. Namely, $X$ has the weak Laplacian principle if, for each $u \in C^2(X)$ bounded above, there exists a sequence $\{x_k\}$ such that
$$ 
u(x_k) > \sup_X u - k^{-1}, \qquad \Delta u(x_k) \le k^{-1},
$$
that is, \eqref{strong_HessianLaplacian} holds with no gradient condition. The weak Hessian principle can be defined accordingly. It turns out that, in most geometric applications, the gradient condition in \eqref{strong_HessianLaplacian} is unnecessary, making thus interesting to study both the possible difference between weak and strong principles, and the geometric conditions guaranteeing the weak principles. 
\begin{remark}[\textbf{Strong Laplacian $\neq$ weak Laplacian}]
\emph{It is easy to construct \emph{incomplete} manifolds satisfying the weak Laplacian principle but not the strong one, for instance $X=\R^m\backslash \{0\}$ (cf. \cite[Ex. 1.21]{maripessoa}). A nice example of a \emph{complete}, radially symmetric surface satisfying the weak Laplacian principle but not the strong one has recently been found in \cite{borbely_counter}. Therefore, the two principles are really different. Also, the weak Laplacian principle is unrelated to the (geodesic) completeness of $X$, and in fact, if one removes a compact subset $K$ that is polar for the Brownian motion on $X$, $X \backslash K$ is still stochastically complete (see Theorem \ref{teo_polar} below).   
}
\end{remark}
\begin{remark}[\textbf{Geometric conditions for weak Laplacian principle}]
\emph{Conditions involving just the volume growth of balls in $X$ (that, by Bishop-Gromov volume comparison, are weaker than those in \eqref{bounds_omoriyau}, cf. \cite{prs}) were first considered in \cite{karp, li}, later improved in \cite[Thm. 9.1]{grigoryan}: $X$ is stochastically complete provided that
\begin{equation}\label{mimi}
\int^{+\infty} \frac{r}{\log \vol B_r}\di r = +\infty. 
\end{equation}
}
\end{remark}

As a consequence of work of R.Z. Khas'minskii \cite{khasminski, prsmemoirs}, $X$ is stochastically complete provided that it supports an exhaustion $w$ outside a compact set $K$ that satisfies
\begin{equation}\label{khasmi_original}
0 < w \in C^2(X \backslash K), \qquad w(x) \ra +\infty \ \text{ as $x$ diverges,} \qquad \Delta w \le \lambda w \ \text{ on } \, X \backslash K,
\end{equation}
for some $\lambda >0$. The analogy with \eqref{strongKhasminskii} and \eqref{eq_evans} is evident, and it is the reason why we call $w$ in \eqref{strongKhasminskii} a Khas'minskii type potential. It was first observed in \cite{marivaltorta} that, in fact, the existence of $w$ satisfying \eqref{khasmi_original} is \emph{equivalent} to the stochastic completeness of $X$. It is therefore natural to ask whether this is specific to the operators $\Delta u$ and $\Delta u - \lambda u$ or if it is a more general fact, and, in the latter case, how one can take advantage from such an equivalence. This is the starting point of the papers \cite{marivaltorta, maripessoa}.\par
%
A further motivation to study Khas'minskii type potentials comes from the desire to understand the link between the Hessian maximum principles and the theory of stochastic processes. It has been suggested in \cite{prs_overview, prs_milan} that a good candidate to be a probabilistic counterpart of a Hessian principle is the \emph{martingale completeness of $X$}. In fact, one can study the non-explosure property  for a natural class of stochastic processes that includes the Brownian motion: the class of martingales (cf. \cite{emery, stroockvaradhan}). 
\begin{itemize}
\item $X$ is called \emph{martingale complete} if and only if each martingale on $X$ has infinite lifetime almost surely. 
\end{itemize}

Differently from the case of stochastic completeness, there is not much literature on the interplay between martingale completeness and geometry, with the notable exception of \cite{emery}. The picture is still fragmentary and seems to be quite different from the Laplacian case: for instance, a martingale complete manifold must be (geodesically) complete (\cite[Prop. 5.36]{emery}). In \cite[Prop. 5.37]{emery}, by using probabilistic tools M. Emery proved that $X$ is martingale complete provided that there exists $w \in C^2(X)$ satisfying
\begin{equation}\label{khasmi_Hessian_original}
\begin{array}{l}
\disp 0 < w \in C^2(X), \qquad w(x) \ra +\infty \ \text{ as $x$ diverges,} \\[0.2cm]
|\nabla w| \le C, \qquad \nabla \di w \le C \metric \ \text{ on } \, X,
\end{array}
\end{equation}
for some $C>0$. Evidently, this is again a Khas'minskii type property. Although the gradient condition in \eqref{khasmi_Hessian_original} might suggest that the martingale completeness of $X$ be related to the strong Hessian principle, in \cite{prs_overview, prs_milan} the authors give some results to support a tight link to the \emph{weak} Hessian principle. Which Hessian principle relates to martingale completeness, and why? Is \eqref{khasmi_Hessian_original} equivalent to the martingale completeness of $X$? \par
The picture described above for Omori-Yau and Ekeland principles, and for parabolicity, stochastic and martingale completeness, suggest that there might be a general ``duality principle" relating an appropriate maximum principle in the form of ($\mathscr{A}$) on open sets, to the existence of suitable Khas'minskii type potentials. This is in fact the case, and the rest of this note aims to settle the problem in the appropriate framework, to explain our main result (the AK-duality) and describe its geometric consequences. An important starting point is to reduce the regularity of solutions of the relevant differential inequalities.

\subsection{On weak formulations: the case of quasilinear operators}

Formulations of maximum principles at infinity for functions with less than $C^2$ regularity have already been studied in depth in recent years, see \cite{prsmemoirs, prsnonlinear, prs_milan, aliasmastroliarigoli, bmpr}, by using distributional solutions. Due to the appearance of quasilinear operators in Geometric Analysis, a natural class of inequalities to investigate is the following quasilinear one:
\begin{equation}\label{quasilinear}
\Delta_a u \doteq \diver \Big( a(|\nabla u|) \nabla u\Big) \ge b(x)f(u)l(|\nabla u|),
\end{equation}
for $a \in C(\R^+)$, $0< b \in C(X)$, $f \in C(\R)$, $l \in C(\R^+_0)$, considered in \cite{bmpr} in full generality. For instance, the study of graphs with prescribed mean curvature and of mean curvature solitons in warped product ambient space leads to inequalities like 
$$
\diver \left( \frac{\nabla u}{\sqrt{1+|\nabla u|^2}}\right) \ge \frac{b(x)f(u)}{\sqrt{1+|\nabla u|^2}}, 
$$
see Section 1 in \cite{bmpr}. The weak and strong maximum principles are stated in terms of functions solving \eqref{quasilinear} on some upper level set, much in the spirit of (i) at page \pageref{eq_stronghessian}, and are summarized in the next 

\begin{definition}\label{def_maxprinc}
We say that
\begin{itemize}
\item $(bl)^{-1}\Delta_a$ satisfies the \emph{weak maximum principle at infinity} if for each non-constant $u \in \lip_\loc(X)$ bounded above, and for each $\eta< \sup_X u$,
$$
\inf_{\{u> \eta\}} \Big\{\Big(b(x)l(|\nabla u|)\Big)^{-1}\Delta_a u\Big\} \le 0,
$$
and the inequality has to be intended in the following sense: if $u$ solves
\begin{equation}\label{def_weakWMP}
\Delta_a u \ge K b(x)l(|\nabla u|) \qquad \text{weakly on } \, \{u>\eta\},
\end{equation}
for some $K \in \R$, then necessarily $K \le 0$.
\item $(bl)^{-1}\Delta_a$ satisfies the \emph{{strong maximum principle at infinity}} if for each non-constant $u \in C^1(X)$ bounded above, and for each $\eta< \sup_X u$, $\eps>0$,
\begin{equation}\label{def_omegaeta}
\Omega_{\eta,\eps} = \{x\in X \ : \ u(x)>\eta, \, \, |\nabla u(x)|<\eps\} \qquad \text{is non-empty,}
\end{equation}
and
$$
\inf_{\Omega_{\eta,\eps}} \Big\{\Big(b(x)l(|\nabla u|)\Big)^{-1}\Delta_a u\Big\} \le 0,
$$
where, again, the inequality has to be intended in the way explained above.
\end{itemize}
\end{definition}

To present, there exist sharp sufficient conditions both to guarantee the weak and the strong principles for $(bl)^{-1}\Delta_a$. These are explicit, and expressed in terms of the growth of the Ricci curvature (of the type in \eqref{bounds_omoriyau}) or of the volume of geodesic balls in $X$ (resembling \eqref{mimi}). The conditions enable to deduce, among others, sharp Liouville theorems for entire graphs with controlled mean curvature. The interested reader is referred to Theorems 1.2 and 1.7 in \cite{bmpr} for the most up-to-date results, and to \cite{prsmemoirs, aliasmastroliarigoli, bmpr} for applications. While working with distributional solutions is quite effective for the weak principle, it seems not an optimal choice in the presence of a gradient condition because $\Omega_{\eta,\eps}$ in \eqref{def_omegaeta} needs to be open and thus forces to restrict to $C^1$ functions $u$. For our purposes, we found more appropriate to work with upper semicontinuous (USC) viscosity solutions.

\section{The general framework} 
We summarize the picture both for the weak and the strong principles. By property ($\mathscr{A}$) above, they can be rephrased in terms of solutions of a fully nonlinear PDE of the type 
\begin{equation}\label{eq_fullynon}
\mathscr{F} \big( x, u(x), \nabla u(x), \nabla^2 u(x) \big) \ge 0 
\end{equation}
on an open subset $U \subset X$, where $\mathscr{F}$ is continuous in its arguments, elliptic and proper, in the following sense:
$$
\begin{array}{lll}
\text{(degenerate ellipticity)} & \quad \mathscr{F}(x,r,p,A) \ge \mathscr{F}(x,r,p,B) & \quad \text{if } \, A \ge B \ \text{ as a quadratic form}, \\[0.2cm]
\text{(properness)} & \quad \mathscr{F}(x,r,p,A) \le \mathscr{F}(x,s,p,A) & \quad \text{if } \, r \ge s.
\end{array}
$$
For instance, 
$$
\begin{array}{ll}
\mathscr{F} = \tr(A) - f(r) & \qquad \text{for the weak Laplacian case, or } \\[0.2cm]
\mathscr{F} = \max\big\{ \lambda_m(A) - f(r), \ |p|- \xi(-r)\big\} & \qquad \text{for the strong Hessian case,} 
\end{array}
$$
for some (any) $f,\xi$ satisfying \eqref{fxi}. Note that the $4$-ple $(x,r,p,A)$ lies in the set
$$
J^2(X) = \Big\{ (x,r,p,A) \ \ : \ \ x \in X, \ r \in \R, \ p \in T_xX, \ A \in \mathrm{Sym}^2(T_xX) \Big\}, 
$$
called the $2$-jet bundle of $X$, and in what follows, with $J^2_xu$ we denote the $2$-jet of $u$ at $x$, i.e. the $4$-ple $(x,u(x),\nabla u(x), \nabla^2 u(x))$. 

\begin{remark}[\textbf{Regularity of solutions}]\label{rem_regularity}
\emph{Although in the above discussion we dealt with $C^2$ solutions, it will be crucial for us to relax the regularity requirements and consider viscosity solutions: an upper semicontinuous (USC) function $u : X \ra [-\infty, +\infty)$  solves \eqref{eq_fullynon} in the viscosity sense provided that, for every $x$ and for every test function $\phi$ of class $C^2$ in a neighbourhood of $x$ and touching $u$ from above at $x$, that is, satisfying 
$$
\left\{ \begin{array}{l}
\phi \ge u \quad \text{around x}, \\
\phi(x) = u(x), 
\end{array}\right.  \qquad \text{it holds} \qquad \mathscr{F}\big( x, \phi(x), \nabla \phi(x), \nabla^2 \phi(x) \big) \ge 0,
$$
}
\end{remark}

To state the property that encompasses the maximum principles discussed above, it is more convenient for us to exploit the geometric approach to fully-nonlinear PDEs pioneered by N.V. Krylov \cite{krylov} and systematically developed by R. Harvey and B. Lawson Jr. in recent years (\cite{HL_primo, HL_dir, HL_existence}). To the differential inequality \eqref{eq_fullynon}, we associate the closed subset 
\begin{equation}\label{Fsube_associated}
F = \Big\{ (x,r,p,A) \ \ : \ \ \mathscr{F}(x,r,p,A) \ge 0 \Big\} \subset J^2(X).
\end{equation}
The ellipticity and properness of $\mathscr{F}$ imply a positivity and negativity property for $F$ (properties $(P)$ and $(N)$ in \cite{HL_dir}). To avoid some pathological behaviour in the existence-uniqueness theory for the Dirichlet problem for $\mathscr{F}=0$, one also needs a mild topological requirement on $F$ (assumption $(T)$ in \cite{HL_dir}). A subset $F \subset J^2(X)$ satisfying $(P),(N),(T)$ is called a \emph{subequation}: it might be given in terms of a function $\mathscr{F} : J^2(X) \ra \R$, as in \eqref{Fsube_associated}, but not necessarily.\par
A function $u \in C^2(X)$ is said to be $F$-subharmonic if $J^2_x u \in F$ for each $x \in X$. If $u \in USC(X)$, as in  Remark \ref{rem_regularity} we say that $u$ is $F$-subharmonic if, for every test function $\phi \in C^2$ at any point $x$, $J^2_x \phi \in F$. Given an open subset $\Omega \subset M$, we define
$$
F(\Omega) = \Big\{ u \in USC(\Omega) \ \ : \ \ \text{$u$ is $F$-subharmonic on $\Omega$}\Big\},
$$
while, for closed $K$, we set $F(K)$ to denote the functions $u \in USC(K)$ that are $F$-subharmonics on $\mathrm{Int}\,K$.\\
Examples that are relevant for us include the following subsets for $f \in C(\R)$ non-decreasing and denoting with $\lambda_1(A) \le \ldots \le \lambda_m(A)$ the eigenvalues of $A$.

\begin{itemize}
\item[$(\EE 1)$] (\textbf{Eikonal type}). The eikonal $E = \big\{|p| \le 1\big\}$, and its modified version $E_\xi = \big\{|p| \le \xi(r)\big\}$ for $\xi$ satisfying \eqref{fxi}. In view of the properties of $\xi$, note that $E_\xi$-subharmonics must be non-positive.\vspace{0.2cm}
\item[$(\EE 2)$] (\textbf{$k$-subharmonics}). $F= \big\{\lambda_1(A)+ \ldots + \lambda_k(A) \ge f(r)\big\}$, $k \le m$. When $f = 0$, $F$-subharmonic functions are called $k$-plurisubharmonic; these subequations, which naturally appear in the theory of submanifolds, have been investigated for instance in \cite{wu,sha, HL_plurisub,HL_existence}. The class encompasses the subequation $\{ \tr(A) \ge f(r)\}$, related both to the stochastic completeness of $X$ (if $f$ satisfies \eqref{fxi}) and to the parabolicity of $X$ (if $f=0$).\vspace{0.2cm}
\item[$(\EE 3)$] (\textbf{Prescribing eigenvalues}). $F= \big\{\lambda_k(A) \ge f(r)\big\}$, for $k \in \{1,\ldots, \dim X\}$. For $k=m$ (and, as we shall see by duality, $k=1$) this is related to the Hessian principles. In particular, if $f \equiv 0$, the subequations 
$$
F_j = \big\{ \lambda_j(A) \ge 0\big\}
$$
describe the $m$-branches associated to the Monge-Amp\`ere equation $\det(\nabla^2 u) = 0$.
\vspace{0.2cm}
\item[$(\EE 4)$] (\textbf{Branches of $k$-Hessian subequation}). For $\lambda \doteq (\lambda_1,\ldots, \lambda_m) \in \R^m$ and $k \in \{1,\ldots, m\}$, consider the elementary symmetric function
$$
\sigma_k(\lambda) = \sum_{1 \le i_1<\ldots < i_k \le m} \lambda_{i_1}\lambda_{i_2}\cdots \lambda_{i_k}.
$$
Since $\sigma_k$ is invariant by permutation of coordinates of $\lambda$, we can define $\sigma_k(A)$ as $\sigma_k$ being applied to the ordered eigenvalues $\{\lambda_j(A)\}$. According to G\"arding's theory in \cite{G}, $\sigma_k(\lambda)$ is a hyperbolic polynomial with respect to the vector $v = (1,\ldots, 1) \in \R^m$. Denote with 
$$
\mu_1^{(k)}(\lambda) \le \ldots \le \mu_k^{(k)}(\lambda) 
$$
the ordered eigenvalues\footnote{That is, the opposite of the roots of $\mathscr{P}(t) \doteq \sigma_k(\lambda + tv)=0$.} of $\sigma_k$. Clearly, $\mu_j^{(k)}$ is permutation invariant, thus the expression $\mu_j^{(k)}(A)$ is meaningful. As a matter of fact,
$$
F_j = \Big\{ \mu_j^{(k)}(A) \ge f(r)\Big\}, \qquad 1 \le j \le k
$$
is a subequation. In particular, if $f \equiv 0$, $F_1, \ldots, F_k$ are called the branches of the $k$-Hessian equation $\sigma_k(\nabla^2 u) = 0$. The smallest branch $F_1$ can be equivalently described as 
$$
F_1 = \Big\{ \sigma_1(A) \ge 0, \ldots, \sigma_k(A) \ge 0 \Big\}.
$$
Many more examples of this kind arise from hyperbolic polynomials $q(\lambda)$, cf. \cite{HL_garding}.
\vspace{0.2cm} 
\item[$(\EE 5)$] \textbf{Subequations on complex, quaternionic and Cayley manifolds}). If $X$ is an almost complex, Hermitian manifold, the complexified Hessian matrix $A$ splits into pieces of type $(2,0),(1,1)$ and $(0,2)$, and it makes sense to consider the last three examples in terms of the eigenvalues of the hermitian symmetric matrix $A^{(1,1)}$. In particular this includes plurisubharmonic functions, that is, solutions of 
$$
\big\{\lambda_1(A^{(1,1)}) \ge 0\big\}.
$$ 
Analogous examples can be given on quaternionic and octonionic manifolds.\vspace{0.2cm}
\item[$(\EE 6)$] (\textbf{Pucci operators}). For $0 < \lambda \le \Lambda$, the Pucci operators (cf. \cite{caffacabre}) are classically defined as
$$
\begin{array}{l}
\disp \mathcal{P}^+_{\lambda, \Lambda}(\nabla^2u) = \sup \Big\{ \tr(X\cdot \nabla^2 u) \ \ : \ \ \text{$X \in \Sym^2(TX)$ with $\lambda I \le X \le \Lambda I$} \Big\}, \\[0.3cm]
\disp \mathcal{P}^-_{\lambda, \Lambda}(\nabla^2u) = \inf \Big\{ \tr(X\cdot \nabla^2 u) \ \ : \ \ \text{$X \in \Sym^2(TX)$ with $\lambda I \le X \le \Lambda I$} \Big\}.
\end{array}
$$ 
The subequations describing solutions of $\mathcal{P}^\pm_{\lambda, \Lambda}(\nabla^2u) \ge f(u)$ can be defined as follows: denoting with $A^+ \ge 0$ and $A^- \le 0$ the positive and negative part of a symmetric matrix $A = A^+ + A^-$, we can set
$$
\begin{array}{l}
F^+_{\lambda,\Lambda} = \Big\{ \lambda \tr(A^-) + \Lambda \tr(A^+) \ge f(r) \Big\}, \\[0.2cm]
F^-_{\lambda,\Lambda} = \Big\{ \Lambda \tr(A^-) + \lambda \tr(A^+) \ge f(r) \Big\}.
\end{array}
$$
\vspace{0.2cm}
\item[$(\EE 7)$] (\textbf{Quasilinear}). We can also consider viscosity solutions of 
\begin{equation}\label{eq_quasilinear}
\Delta_a u \doteq \mathrm{div}\big(a(|\nabla u|)\nabla u\big) \ge f(u),
\end{equation}
for $a \in C^1(\R^+)$ satisfying 
\begin{equation}\label{ipo_a}
\theta_1(t) \doteq a(t) + ta'(t) \ge 0, \qquad \theta_2(t) \doteq a(t)>0.
\end{equation}
Examples include 
\begin{itemize}
\item[-] the mean curvature operator, describing the mean curvature of the graph hypersurface $\{(x, v(x)) : x \in M\}$ into the Riemannian product $M \times \R$. In this case, $a(t) = (1+t^2)^{-1/2}$;
\item[-] the $q$-Laplacian $\Delta_q$, $q>1$, where $a(t) = t^{q-2}$;
\item[-] the operator of exponentially harmonic functions, where $a(t) = \exp\big(t^2\big)$, considered for instance in \cite{DE};
\end{itemize}
Indeed, expanding the divergence we can set 
\begin{equation*}
F = \overline{\left\{ p \neq 0, \ \tr \big( T(p)A\big) > f(r) \right\}}, 
\end{equation*}
where 
\begin{equation*}
T(p) \doteq a(|p|)\metric + \frac{a'(|p|)}{|p|} p \otimes p = \theta_1(|p|) \Pi_p + \theta_2(|p|) \Pi_{p^\perp},
\end{equation*}
and $\Pi_p, \Pi_{p^\perp}$ are, respectively, the $(2,0)$-versions of the orthogonal projections onto the spaces $\langle p\rangle$ and $p^\perp$. Similarly, we can consider the non-variational, \emph{normalized} quasilinear operator given by 
$$
F = \overline{\left\{ p \neq 0, \ \frac{\tr \big( T(p)A\big)}{\max\{\theta_1(|p|), \theta_2(|p|)\}} > f(r) \right\}}.
$$
In the case of the mean curvature operator, the last subequation represents viscosity solutions of 
$$
\diver \left( \frac{\nabla u}{\sqrt{1+|\nabla u|^2}} \right) \ge \frac{f(u)}{\sqrt{1+ |\nabla u|^2}},
$$
that are related to prescribed mean curvature graphs and mean curvature solitons in warped product spaces, see \cite[Chapter 1]{bmpr}. \vspace{0.2cm}

\item[$(\EE 8)$] (\textbf{$\infty$-Laplacian}). The normalized $\infty$-Laplacian $F= \overline{\big\{p \neq 0, \ |p|^{-2}A(p,p) > f(r)\big\}}$.\vspace{0.2cm}
\end{itemize}

In $(\EE 7)$ and $(\EE 8)$, the necessity to take as $F$ the closure of its interior is made necessary to match property $(T)$, due to the possible singularity of the operator at $p=0$. The above examples can be defined on each Riemannian (complex, quaternionic, octonionic) manifold, since there is no explicit dependence of $F$ from the point $x$, and are therefore called \emph{universal} subequations. To include large classes of subequations with coefficients depending on the point $x$, that can be seen as ``deformations" of universal ones, Harvey and Lawson in \cite{HL_dir} introduced the concept of \emph{local jet-equivalence} between subequations. Without going into the details here, we limit to say that, for instance, any semilinear inequality of the type
$$
a^{ij}(x) u_{ij} + b^i(x) u_i \ge c(x) f(u)
$$
for smooth $a^{ij}, b^i,c$ with $c>0$ on $X$ and $[a^{ij}]$ positive definite at every point, is locally jet-equivalent to the universal example describing solutions of 
$$
\Delta u \ge f(u).
$$
The key fact is that local jet-equivalence allows to transfer properties from the universal example to the subequations locally jet equivalent to it.\par
Supersolutions for $\mathscr{F}$, that is, solutions of 
\begin{equation}\label{eq_supersolution}
\mathscr{F}\big( x,u(x),\nabla u(x), \nabla^2 u(x) \big) \le 0,
\end{equation}
are taken into account starting from the observation that $w = -u$ solves
\begin{equation}\label{eq_menosupersolution}
\widetilde{\mathscr{F}}\big( x,w(x),\nabla w(x), \nabla^2 w(x) \big) \ge 0, 
\end{equation}
with 
$$
\widetilde{\mathscr{F}}(x,r,p,A) = -\mathscr{F}(x,-r,-p,-A).
$$
This suggests to define the \emph{dual subequation}
$$
\widetilde{F} = - \sim \mathrm{Int}(F).
$$
In particular,
$$
\text{if} \quad  F = \big\{ \mathscr{F}(x,r,p,A) \ge 0\big\}, \qquad \text{then} \qquad \widetilde{F} = \big\{ \widetilde{\mathscr{F}}(x,r,p,A) \ge 0\big\}.
$$
Therefore, $u$ is $\widetilde{F}$-subharmonic if $-u$ is a supersolution in the standard, viscosity sense, and $u$ is $F$-harmonic on $\Omega$ if $u \in F(\Omega)$ and $-u \in \widetilde{F}(\Omega)$. The above operator is in fact a duality, in particular 
\begin{equation}\label{duality}
\widetilde{F\cap G} = \widetilde F \cup \widetilde G, \qquad \widetilde{\widetilde{F}} = F  
\end{equation}
for each subequations $F,G$, and $\widetilde{F}$ is a subequation if $F$ is so. Concerning examples $(\EE1)$ to $(\EE 8)$, 
\begin{itemize}
\item In $(\EE 1)$, the dual of the eikonal equation is $\big\{|p| \ge 1\big\}$, that of $E_\xi$ is $\widetilde{E_\xi} = \big\{ |p| \ge \xi(-r)\big\}$.
\item In $(\EE 2)$, 
$$
\widetilde{F} = \left\{\sum_{j=m-k+1}^m \lambda_j(A) \ge f(r)\right\};
$$
in particular, $\big\{\tr(A) \ge f(r)\big\}$ is self-dual: $\widetilde{F} = F$.
\item In $(\EE 3)$, $\widetilde{F} = \big\{ \lambda_{m-k+1}(A) \ge f(r)\big\}$. 
\item In $(\EE 4)$, if $F = \big\{ \mu_j^{(k)}(A) \ge f(r)\big\}$ then $\widetilde{F} = \big\{ \mu^{(k)}_{k-j+1}(A) \ge f(r)\big\}$.
\item In $(\EE 6)$, the dual of $F^\pm_{\lambda, \Lambda}$ is $F^\mp_{\lambda, \Lambda}$.
\item Examples $(\EE7)$ and $(\EE 8)$ are self-dual: $\widetilde{F} = F$.
\end{itemize}

Given a subequation $F$ and for $g \in C(X)$ we shall introduce the \emph{obstacle} subequation
$$
F^g = F \cap \big\{ r \le g(x)\big\},
$$
that describes $F$-subharmonic functions lying below the obstacle $g$. Note that, since the dual of $\big\{ r \le g(x)\big\}$ is $\{ r \le -g(x)\}$, $\widetilde{F^g}$ describes functions $u$ that are $\widetilde{F}$-subharmonic on the upper set $\{ u(x) > g(x)\}$. As we shall see in a moment, functions in $\widetilde{F^0}$ will be used to describe the property that unifies the maximum principles at infinity described above, and it is therefore expectable, by duality, that obstacles subequations play an important role in our main result.

\begin{definition}
Let $F \subset J^2(X)$ be a subequation. Given $\Omega \Subset X$ open, $g \in C(\overline \Omega)$ and $\phi \in C(\partial \Omega)$ with $\phi \le g$ on $\partial \Omega$, a function $u \in C(\overline{\Omega})$ is said to solve the \emph{obstacle problem} with obstacle $g$ and boundary value $\phi$ if
$$
\left\{ \begin{array}{l}
u \quad \text{is $F^g$-harmonic on $\Omega$} \\[0.2cm]
u = \phi \quad \text{on } \, \partial \Omega.
\end{array}\right.
$$
\end{definition}

%

%
%
%
%
%

\section{Ahlfors, Khas'minskii properties and the AK-duality}\label{sec_teoprinci}

The definition of the next property is inspired by the original work of Ahlfors \cite{ahlforssario}, as well as by the recent improvements in \cite{aliasmastroliarigoli,aliasmirandarigoli, imperapigolasetti}. 

\begin{definition}\label{def_ahlfors}
A subequation $H\subset J^2(X)$ is said to satisfy the \emph{Ahlfors property} if, having set $H_0 = H \cup \{r \le 0\}$, for each $U \subset X$ open with non-empty boundary and for each $u \in H_0(\overline{U})$ bounded from above and positive somewhere, it holds
$$
\sup_{\partial U} u^+ \equiv \sup_{\overline{U}} u.
$$
\end{definition}

Roughly speaking, when $u$ is $H$-subharmonic on the set $\{u>0\}$, the Ahlfors property means that its supremum is attained on the boundary of $U$. 

\begin{example}\label{ex_importante}
\emph{We consider the following subequations: 
\begin{itemize}
\item[1)] If $F = \{ \tr(A) \ge 0\}$, in view of Ahlfors' characterization \cite{ahlforssario}, the Ahlfors property for $\widetilde{F}$ ($=F$) is a version, for viscosity solutions, of the property characterizing the parabolicity of $X$ in \cite{ahlforssario}. Similarly, if $F = \{\tr(A) \ge f(r)\}$ for $f$ satisfying \eqref{fxi}, the Ahlfors property for $\widetilde{F}$ can be seen as a viscosity version of the weak Laplacian principle, that is, of the stochastic completeness of $X$. As a matter of fact (cf. \cite{maripessoa}), in both of the cases the property still \emph{characterizes} the parabolicity, respectively the stochastic completeness, of $X$.
\item[2)] The Ahlfors property for the dual eikonal $\widetilde{E} = \big\{ |p| \ge 1\big\}$ can be viewed as a viscosity version of Ekeland's principle. Its equivalence to the original Ekeland's principle, hence to geodesic completeness, is one of the applications of our main result below.
\item[3)] Consider the subequations $F = \big\{ \lambda_1(A) \ge f(r)\big\}$ and $E_\xi = \big\{ |p| \le \xi(r)\big\}$, for $(f,\xi)$ satisfying \eqref{fxi}. Then, the Ahlfors property for the dual
$$
\begin{array}{lcl}
\disp \widetilde{ F \cap E_\xi} & = & \disp \widetilde{F} \cup \widetilde{E_\xi} \ \ = \ \ \big\{ \lambda_m(A) \ge f(r)\big\} \cup \big\{ |p| \ge \xi(-r)\big\} \\[0.3cm]
& = & \disp \Big\{ \max\big\{ |p| - \xi(-r), \ \lambda_m(A) - f(r) \big\} \ge 0\Big\}
\end{array}
$$
can be seen as a viscosity analogue of the strong Hessian principle. Analogously, the Ahlfors property for $\widetilde{F} \cup \widetilde{E_\xi}$ with $F = \big\{ \tr(A) \ge f(r)\}$ is a natural, viscosity version of Yau's strong Laplacian principle. Differently from the examples in 1), it is \emph{not known} whether these Ahlfors properties are, in fact, equivalent to the classical strong Hessian and Laplacian principles for $C^2$ solutions.   
\end{itemize}
}
\end{example}
A comment is in order: although the above viscosity versions might be strictly stronger than the corresponding classical ones, all of the known geometric conditions that guarantee the weak and strong Hessian or Laplacian principles in the $C^2$ case \emph{also ensure} their viscosity counterparts. Therefore, passing to the viscosity realm does not prevent from geometric applications, and indeed is able to uncover new relations. To see them, we shall introduce the Khas'minskii properties, that generalize \eqref{strongKhasminskii} and \eqref{khasmi_original}. Hereafter, a pair $(K,h)$ consists of 
\begin{itemize}
\item[-] a smooth, relatively compact open set $K \subset X$;
\vspace{0.1cm} 
\item[-] a function $h \in C(X\backslash K)$ satisfying $h < 0$ on $X\backslash K$ and $h(x) \ra -\infty$ as $x$ diverges. 
\end{itemize}

\begin{definition}\label{def_khasmi} A subequation $F\subset J^2(X)$ satisfies the \emph{Khas'minskii property} if, for each pair $(K,h)$, there exists a function $w$ satisfying:
\begin{equation}
\begin{array}{l}
\disp w \in F(X\backslash K), \quad h \le w \leq 0 \quad \text{on } \, X \backslash K, \quad \text{and} \quad w(x) \ra -\infty \quad \text{ as $x$ diverges.}
\end{array}
\end{equation}
Such a function $w$ is called a \emph{Khas'minskii potential} for $(K,h)$.
\end{definition}

Loosely speaking, $F$ has the Khas'minskii property if it is possible to construct $F$-subharmonic exhaustions that decay to $-\infty$ as slow as we wish. In practice, checking the Khas'minskii property  might be a hard task, and often, from the geometric problem under investigation, one is just able to extract \emph{some} of the Khas'minskii potentials. This motivates the following definition (cf. the recent \cite{mmr_berestycki}). 
\begin{definition}\label{def_weakkhasmi}
A subequation $F\subset J^2(X)$ satisfies the \emph{weak Khas'minskii property} if there exist a relatively compact, smooth open set $K$ and a constant $C \in \R \cup \{+\infty\}$ such that, for each $x_0 \not \in \overline{K}$ and each $\eps>0$, there exists $w$ satisfying 

\begin{equation}\label{eq_WK}
\begin{array}{l}
\disp w \in F(X\backslash K), \qquad w \leq 0 \quad \text{ on } \, X \backslash K, \qquad w(x_0) \ge -\eps, \qquad \disp \limsup_{x \ra \infty} w(x) \le -C. 
\end{array}
\end{equation}
We call such a $w$ a \emph{weak Khas'minskii potential} for the triple $(\varepsilon,K,\{x_0\})$.
\end{definition} 
\begin{remark}
\emph{When $C= +\infty$ and $F$ is scale invariant (that is, it is fiber-wise a cone), the condition over $\eps$ in \eqref{eq_WK} can be avoided by a simple rescaling. 
}
\end{remark} 

\begin{example}
\emph{The existence of $w$ in \eqref{strongKhasminskii}, for $G$ satisfying \eqref{ipo_G_SMP}, implies that a weak Khas'minskii property holds for the subequation
$$
\big\{ \tr(A) \ge G(-r) \big\} \cap \big\{ |p| \le G(-r) \big\}.
$$
Indeed, the weak Khas'minskii potentials can be constructed by suitably rescaling and modifying the function  $-w$. Up to playing with $w$ and $G$, the above is equivalent to the weak Khas'minskii property for 
$$
\big\{ \tr(A) \ge f(r) \big\} \cap \big\{ |p| \le \xi(r) \big\},
$$
for some (any) $(f,\xi)$ satisfying \eqref{fxi}. Similarly, the existence of $w$ in \eqref{khasmi_original} implies the weak Khas'minskii property for $F = \big\{\tr(A) \ge f(r)\big\}$.
}
\end{example}

We are ready to state our main result, the Ahlfors-Khas'minskii duality (shortly, AK-duality), Theorems 4.3 and 4.10 in \cite{maripessoa}. It applies to subequations $F$ on $X$ that are locally jet-equivalent to a universal one and satisfy a few further assumptions. Some of them are merely technical and will not be described here. Their  validity characterizes the set of \emph{admissible} subequations, that is still quite general. For instance, each of the examples in $(\EE 2), \ldots, (\EE 6)$, and the subequations locally jet-equivalent to them, are admissible provided that $f$ satisfies \eqref{fxi}.

\begin{theorem}\label{teo_main}
Let $F \subset J^{2}(X)$ be an admissible subequation, locally jet-equivalent to a universal one. Assume that 
\begin{itemize}
\item[$(\HH 1)$] $\ $ negative constants are strictly $F$-subharmonic; \vspace{0.1cm}
\item[$(\HH 2)$] $\ $ $F$ satisfies the comparison theorem: whenever $\Omega \Subset X$ is open, $u \in F(\Omega)$, $v \in \widetilde{F}(\Omega)$, 
$$ 
u+v \leq 0 \quad \text{on } \ \partial \Omega \quad \Longrightarrow \quad u + v \leq 0 \quad \text{on } \ \Omega.
$$
\end{itemize}
Then, AK-duality holds for $F$ and for $F \cap E_\xi$ for some (any) $\xi$ satisfying \eqref{fxi}, i.e.,
\begin{equation*}
\begin{array}{c}
F \ \mbox{ satisfies $(K)$} \\
(\mbox{Khas'minskii prop.})
\end{array}
\quad \Longleftrightarrow  
\quad 
\begin{array}{c}
F \ \mbox{ satisfies $(K_\weak)$} \\
(\mbox{weak Khas'minskii prop.})
\end{array}
\quad \Longleftrightarrow  
\quad 
\begin{array}{c}
\widetilde{F} \ \mbox{ satisfies $(A)$} \\
(\mbox{Ahlfors prop.}) 
\end{array},
\end{equation*}
and 
$$
F \cap E_\xi \ \mbox{ satisfies $(K)$} \quad \Longleftrightarrow \quad F \cap E_\xi \ \mbox{ satisfies $(K_\weak)$} \quad \Longleftrightarrow \quad \widetilde{F} \cup \widetilde{E_\xi} \ \mbox{ satisfies $(A)$} .
$$
\end{theorem}

Seeking to clarify the role of each assumption in the AK-duality, we briefly examine the importance of each one. 

\begin{remark}[\textbf{On assumption $(\HH 1)$}] 
\emph{This property holds for each of $(\EE 2), \ldots, (\EE 6)$ provided that $f$ satisfies \eqref{fxi}, and it is important to ensure the validity of the finite maximum principle: functions $u \in \widetilde{F^0}(Y)$ cannot achieve a local positive maximum. The latter is crucial for our proof to work. 
}
\end{remark}
\begin{remark}[\textbf{On assumption $(\HH 2)$}] 
\emph{This is delicate to check, and curiously enough it plays a role just in the proof of $(K_\weak) \Rightarrow (A)$. Comparison holds for uniformly continuous subequations which are strictly increasing in the $r$ variable, a case that covers examples $(\EE 2), \ldots, (\EE 5)$ as well as $(\EE 6)$, see \cite[Thm. 2.25]{maripessoa} for details\footnote{In \cite{maripessoa}, the uniform continuity of the Pucci operators in $(\EE 6)$ is not explicitly stated but can be easily checked. For instance, in the case of $\mathcal{P}^+_{\lambda,\Lambda}$, referring to Definition 2.23 in \cite{maripessoa} and using the  min-max definition, 
$$
\mathcal{P}^+_{\lambda,\Lambda}(B) \ge \disp \mathcal{P}^+_{\lambda,\Lambda}(A) - \mathcal{P}^+_{\lambda,\Lambda}(A-B) \ge \mathcal{P}^+_{\lambda,\Lambda}(A) - \Lambda \tr\big((A-B)_+\big). 
$$
If $\|(A-B)_+\| < \delta$, then $\mathcal{P}^+_{\lambda,\Lambda}(B) \ge \mathcal{P}^+_{\lambda,\Lambda}(A) - m\Lambda \delta$, that proves the uniform continuity of $F_{\lambda,\Lambda}^+$. The case of $F_{\lambda, \Lambda}^-$ is analogous.}. The uniform continuity resembles the classical condition 3.14 in \cite{cil}. Regarding examples $(\EE 7)$ and $(\EE 8)$, the worse dependence on the gradient term makes comparison much subtler. One can check the comparison theorem for the universal subequation in $(\EE 8)$, even with $f \equiv 0$, by means of other interesting methods, cf. \cite[Thm. 2.27]{maripessoa}. As for $(\EE 7)$, on Euclidean space the validity of comparison for strictly increasing $f$ is a direct application of the classical theorem on sums (i.e. Ishii's Lemma, \cite{cil} and \cite[Thm. C.1]{HL_dir}).  In a Riemannian setting, Ishii's Lemma uses the infimal convolutions with the squared distance function $r^2(x,y)$ on $X \times X$, and in neighbourhoods where the sectional curvature is negative the Hessian of $r^2$ is \emph{positive} to second order on pairs of parallel vectors. The error term produced by such positivity can be easily controlled for normalized quasilinear operators, but in the unnormalized case one has to require the boundedness of the eigenvalues $\theta_1, \theta_2$ in order to avoid further a-priori bounds on the subsolutions and supersolutions (like Lipschitz continuity of either one of them). Nevertheless, we note that the boundedness of $\theta_1,\theta_2$ notably includes the mean curvature operator.
}
\end{remark}

Summarizing, we have

\begin{corollary}\label{cor_main}
The AK-duality holds both for $F$ and for $F \cap E_\xi$, with $\xi$ satisfying \eqref{fxi}, in each of the following cases:
\begin{itemize}
\item[-] $F$ is locally jet-equivalent to any of $(\EE 2), \ldots, (\EE 6)$;
\item[-] $F$ is the normalized quasilinear example in $(\EE 7)$, or $F$ is the unnormalized example with $\theta_1,\theta_2 \in L^\infty(\R^+)$;
\item[-] $F$ is the universal example in $(\EE 8)$,
\end{itemize}
and, in each case, $f$ satisfies \eqref{fxi}. Furthermore, the AK-duality holds for the eikonal subequations in $(\EE 1)$.
\end{corollary}

\begin{proof}[{\bf Sketch of the proof of the AK-duality:}] Since $(K) \Rightarrow (K_\weak)$ is obvious, we shall prove $(K_w) \Rightarrow (A)$ and $(A) \Rightarrow (K)$. The proof of the first implication is inspired by a classical approach that dates back to Phr\'agmen-Lindeloff type theorems in classical complex analysis, and we therefore concentrate on $(A) \Rightarrow (K)$.\\
Fix a pair $(K,h)$, and a smooth exhaustion $\{D_j\}$ of $X$ with $K \subset D_1$. Our desired Khas'minskii potential $w$ will be constructed as a locally uniform limit of a decreasing sequence of USC functions $\{w_i\}$, such that $w_0=0$ and for each $i \ge 1$ we have:
\begin{equation}\label{cond_induction_w}
\begin{array}{ll}
(a) & \quad w_i \in F(X\backslash K), \qquad w_{i} = (w_i)_* = 0 \quad \text{ on } \, \partial K; \\[0.2cm]
(b) & \quad w_i \geq -i \quad \text{ on } \, X\backslash K, \qquad w_i = -i \quad \text{outside a compact set $C_{i}$ containing $D_{i}$;} \\[0.2cm]
(c) & \quad \left( 1- 2^{-i-2}\right)h < w_{i+1} \le w_i \le 0 \quad \text{on } \ X \backslash K, \qquad \|w_{i+1}-w_i\|_{L^{\infty}(D_{i}\backslash K)} \leq \frac{\varepsilon}{2^{i}}.
\end{array}
\end{equation}
With the above properties, the sequence $\{w_i\}$ is locally uniformly convergent on $X\backslash K$ to some function $w \in F(X\backslash K)$ with $h \le w \le 0$ on $X \backslash K$ and satisfying $w(x) \ra -\infty$ as $x$ diverges, that is, to the desired Khas'minskii potential. 

Fix $w=w_i$. We build $w_{i+1}$ inductively via a sequence of obstacle problems, an idea inspired by \cite{valtorta, marivaltorta}: we fix obstacles $g_j = w + \lambda_j$, for some sequence $\{\lambda_j\} \subset C(X\backslash K)$ such that
\begin{equation}\label{def_hj}
\left\lbrace
\begin{array}{l}
\disp 0 \ge \lambda_j \ge -1, \quad \lambda_j = 0 \ \text{ on } \, K, \quad \lambda_j =-1 \ \text{ on } \, X \backslash D_{j-1}, \\[0.1cm]
\disp \text{$\{\lambda_j\}$ is an increasing sequence, and $\lambda_j \uparrow 0$ locally uniformly,}
\end{array}\right.
\end{equation}
and search for solutions of the obstacle problem
\begin{equation}\label{obstacle_uj_full}
\left\{ \begin{array}{l}
u_j \qquad \text{is $F^{g_j}$-harmonic on $D_j \backslash K$}, \\[0.2cm]
u_j = 0 \quad \text{on } \, \partial K, \qquad u_j = -i-1 \quad \text{on } \, \partial D_j.
\end{array}\right.
\end{equation}
However, in some relevant cases we cannot fully solve \eqref{obstacle_uj_full}. The first problem we shall consider is the absence of barriers, needed to prove the existence of $u_j$ via Perron's method. No problem arise on $\partial D_j$, since the constant function $-i-1$ is $F^{g_j}$-subharmonic by $(\HH 1)$. However, since we are working in the complement of a compact set $K$ (think of $K$ being a small geodesic ball, for instance), $\partial K$ might be concave in the outward direction, that in general prevents from having barriers there. To overcome this problem, we modify $X$ inside of $K$ by gluing a compact manifold $Y$ that is Euclidean in a sufficiently small ball $\mathbb{B}$. The gluing only involves small annuli inside of $\mathbb{B}$ and $K$, with the new metric coinciding with those of $X$ and $Y$ outside of the gluing region. In particular, the new manifold is Euclidean in a neighbourhood of $\partial \mathbb{B}$. In this way, replacing $K$ by $K' = Y \backslash \mathbb{B}$, $X \backslash K$ embeds isometrically into $X \backslash K'$ and the latter has a \emph{convex} boundary isometric to $\partial \mathbb{B}$. Because of a technical assumption included in those defining the admissibility of $F$, this is enough to produce barriers on $\partial K'$. Once we perform this change, we suitably modify the subequation $F$ preserving it outside the gluing region, and making it, on $K'$, the universal Riemannian subequation to which $F$ is locally jet equivalent. Although these modifications change in several ways the manifold and the subequation, they are stable to preserve the Ahlfors property for $\widetilde F$ as well as the assumption $(\HH 1)$. The price to pay is that we may lose the comparison property $(\HH 2)$, since comparison is very sensitive to the geometry of $X$, at least for some relevant operators like those in $(\EE 2), (\EE 3), (\EE 4)$. This is the main reason why, generally, we cannot fully solve \eqref{obstacle_uj_full}. However, with barriers finally available, Perron's method yields an ``almost solution" $u_j$ of the obstacle problem on $X \backslash K'$ (see \cite{HL_dir}, \cite[Thm. 3.3]{maripessoa}), that is, $u_j$ solves 
\begin{equation}\label{dirichlet_step1}
\left\{
\begin{array}{ll}
u_{j} \in F^{g_{j}}(\overline{D_{j}\backslash K'}), \quad (-u_{j})^* \in \widetilde{F^{g_{j}}}(\overline{D_{j}\backslash K'}),\\[0.1cm]
u_{j} = (u_{j})_* = 0 \qquad \mbox{on} \ \ \partial K' ,\\[0.1cm]
u_{j} = (u_{j})_* = -i-1 \qquad \mbox{on} \ \ \partial D_{j}.
\end{array}
\right.
\end{equation}

We extend $u_j$ outside $D_j$ by setting $u_{j} \doteq -i-1$, and define $v_{j} \doteq (-u_{j})^* -i$. By the definition of Perron's solution, the sequence $\{v_j\}$ is decreasing on $X \backslash K'$. Thus, passing to the limit using that $g_j \ra w \ge -i$ as $j \ra \infty$, $w = -i$ outside of $C_i$, we get 
$$
v_j \downarrow v \in \widetilde{F^{0}}( X\backslash K'), \quad \text{with} \quad \left\{\begin{array}{ll}
-i \le v \le 1 & \quad \text{on } \, X \backslash K', \\[0.1cm]
v = -i < 0 & \quad \text{on } \, \partial K', \\[0.1cm]
v \ge 0 & \quad \text{on } \, X \backslash C_{i}.
\end{array}\right.
$$
Here is the crucial point where the Ahlfors property enters: in fact, using Ahlfors on $X \backslash K'$ we infer that $v \equiv 0$ outside of $C_i$, and by the USC-version of Dini's theorem, 
$$
v_j \downarrow 0 \qquad \text{locally uniformly on } \, X \backslash C_{i}.
$$
Then, the definition of $v_j$ yields 
\begin{equation}\label{conve_uj}
u_j \uparrow -i \quad \text{locally uniformly on } \, X \backslash C_i.
\end{equation}
It remains to investigate the convergence of $u_j$ on the bounded set $\overline C_i\backslash K'$. Although comparison might fail on this set, what guarantees the convergence $u_j \uparrow w$ is that each $u_j$, being a Perron's solution, is maximal in the set of $F^{g_j}$-subharmonic functions whose boundary values do not exceed $0$ (on $\partial K'$) and $-i-1$ (on $\partial D_j$). Concluding, $u_j \uparrow w$ locally uniformly on $X\backslash K'$, hence on $X \backslash K$. For $j$ large enough, if we set $w_{i+1} = u_j$ it is therefore possible to meet all of $(a),(b),(c)$ in \eqref{cond_induction_w}, as desired.\par
To treat the case when $F$ is coupled to the eikonal equation $E_\xi$, the issue is again the absence of barriers on $\partial K'$ to solve the obstacle problem for $F^{g_j} \cap E_\xi$. Indeed, even if, after the gluing, $\partial K'$ is convex in the direction pointing towards $X \backslash K'$, barriers must be $E_\xi$-subharmonic and the gradient control may prevent to build barriers up to height $-i$ at step $i$. To overcome this problem, the idea is to modify the subequation $E_\xi$ in the gluing region in a different way at each step $i$, weakening the bound $\xi(r)$ by means of a cut-off function $\phi_i$ supported in a neighbourhood of $\partial K'$. The size of $\phi_i$ depends on the $L^\infty$ norm of the gradient of the barriers on $\partial K'$ joining zero to $-i$, and therefore it diverges as $i \ra \infty$. In this way, we clearly lose the gradient control in the limit in a neighbourhood of $K'$, but since $K' \Subset K$, for suitable $\phi_i$ no property of $w_i$ on $X \backslash K$ get lost.
\end{proof}

It is worth to remark that an important case was left uncover by Theorem \ref{teo_main}. For instance, when $F$ is independent on $r$ (examples $(\EE 2), \ldots, (\EE 6)$ with $f \equiv 0$), assumption $(\HH 1)$ does not hold. However $(\HH 1)$ is just used to ensure the strong maximum principle for functions in $\widetilde{F^0}$ on any manifold. Therefore, we can state the following alternative version of our main theorem.

\begin{theorem}\label{teo_main_semH1}
Let $F \subset J^{2}(X)$ be a universal subequation satisfying $(\HH 2)$ and
\begin{itemize}
\item[$(\HH 1')$] $\widetilde{F}$ has the strong maximum principle on each manifold $Y$ where it is defined: $\widetilde{F^0}$-subharmonic functions on $Y$ are constant if they attain a local maximum. 
\end{itemize}
Then, AK-duality holds for $F$.\footnote{Theorem \ref{teo_main_semH1} can be stated for $F$ locally jet-equivalent to a universal example, provided that the strong maximum principle in $(\HH 1')$ holds for each manifold $Y$ and each $\widetilde{F} \subset J^2(Y)$ constructed by gluing as in the theorem.} 
\end{theorem}

The strong maximum principle for viscosity subsolutions is a classical subject that has been investigated by many authors, in particular we quote \cite{HL_SMP, bardidalio, kawohlkutev_SMP} (cf. also \cite{PuRS, pucciserrin, bmpr} for the quasilinear case). Particularizing Theorems \ref{teo_main} and \ref{teo_main_semH1} to the mean curvature operator and its normalized version, for which the strong maximum principle is proved in \cite{kawohlkutev_SMP}, we have the following:

\begin{theorem}\label{teo_meancurv}
The AK-duality holds for the subequation in $(\EE 7)$ describing solutions of 
\begin{equation}\label{eq_bonito_MC}
\diver \left( \frac{\nabla u}{\sqrt{1+|\nabla u|^2}}\right) \ge f(u) \qquad \text{and} \qquad \diver \left( \frac{\nabla u}{\sqrt{1+|\nabla u|^2}}\right) \ge \frac{f(u)}{\sqrt{1+|\nabla u|^2}},
\end{equation}
for every non-decreasing, odd function $f \in C(\R)$.
\end{theorem}

\noindent{\bf Other quasilinear operators.}\vspace{0.1cm}

As said, the lack of a strong enough comparison theorem forces us to require, in the unnormalized version of $(\EE 7)$ of Corollary \ref{cor_main}, the boundedness of the eigenvalues $\theta_1,\theta_2$ on $\R^+_0$. We believe that the AK-duality holds for each subequation locally jet-equivalent to $(\EE 7)$, both normalized and unnormalized, independently of $\theta_1,\theta_2$. More information can be found in Section 2.5 and Appendix A of \cite{maripessoa}, where the authors investigate classes of quasilinear operators where comparison holds. For inequalities of the type
$$
\diver \mathscr{A}(x, \nabla u) \ge \mathscr{B}(x,u),
$$
with $\mathscr{A}$ a Carathe\'odory map that locally behaves like a $q$-Laplacian, and $\mathscr{B}$ non-decreasing in $u$ with $u \mathscr{B}(x,u) \ge 0$, the AK-duality in a slightly less general version was first established in \cite{marivaltorta}. The use of weak instead of viscosity solutions allows to work with very general $\mathscr{A},\mathscr{B}$, since a comparison theorem is easy to show and the obstacle problem is solvable by classical results. Nevertheless, the method does not allow to include a gradient dependence and thus investigate the ``strong" versions of the corresponding Ahlfors property.
\vspace{0.2cm}

\noindent{\bf Liouville property.}\vspace{0.1cm}

As the cases of parabolicity and stochastic completeness show, the Ahlfors property is also related to the next Liouville one:
\begin{definition}
A subequation $F \subset J^2(X)$ has the \emph{Liouville property} if any $u \in F(X)$ bounded from above and non-negative is constant.
\end{definition}

Indeed, in \cite{marivaltorta} the main result itself is expressed as a duality between Khas'minskii and Liouville properties. It is not difficult to show that the Ahlfors property implies the Liouville one, and that the two are equivalent provided that 
\begin{equation}\label{uharmo}
u \equiv 0 \qquad \text{is $F$-harmonic},
\end{equation}
cf. \cite[Prop. 4.2]{maripessoa} and previous work in \cite{ahlforssario, grigoryan, aliasmastroliarigoli, aliasmirandarigoli}. While \eqref{uharmo} holds in many instances, there are notable exceptions, for example the eikonal subequation. For such subequations, it is the Ahlfors property the one that actually realizes duality.

\section{Applications}

\subsection{Completeness, viscosity Ekeland principle and $\infty$-parabolicity}

Let $u \in C^1(X)$ be a function bounded from above and assume that there exists a classical $C^1$-Khas'minskii potential $w$, that is, satisfying only the ehxaustion and the gradient properties in \eqref{khasmi_Hessian_original}. Up to a rescaling, $w$ is a Khas'minskii potential for the eikonal subequation $E = \{\vert p\vert \leq 1\}$. Following the original argument that goes back to Ahlfors \cite{ahlfors}, we consider a sequence of functions $u + \frac{1}{k}w$ each of which attains a maximum at some point $x_k \in X$. Up to choosing a subsequence, it is easy to see that 
$$ 
u(x_k) > \sup_{X} u - \frac{1}{k}, \quad \text{and} \quad  u(y) \leq u(x_k) + \frac{1}{k}d(x_k,y) \quad \text{for } \ y \ \text{nearby } \ x_k.
$$ 
Thus, recalling the AK-duality, one can see the Ahlfors property for the dual eikonal subequation $\widetilde{E} = \{|p| \ge 1\}$ as a sort of \emph{viscosity version of Ekeland principle.} Clearly, the above argument does not give a formal proof of the equivalence between the Ahlfors property for $\widetilde E$ and the Ekeland principle stated in Definition \ref{def_ekeland_classico}. In fact, it follows from the next application of Theorem \ref{teo_main}: 

\begin{theorem}\label{teo_ekeland_intro}
Let $X$ be a Riemannian manifold. Then, the following statements are equivalent:
\begin{itemize}
\item[(1)] $X$ is complete.
\item[(2)] the dual eikonal $\widetilde E= \{|p|\ge 1\}$ has the Ahlfors property (viscosity Ekeland principle).
\item[(3)] the infinity Laplacian $F_\infty \doteq \overline{\{A(p,p)>0\}}$ has the Ahlfors property.
\item[(4)] $F_\infty$ has the next strengthened Liouville  property:
\begin{quote}
Any $F_\infty$-subharmonic function $u \ge 0$ such that $|u(x)| = o\big(\varrho(x)\big)$ as $x$ diverges ($\varrho(x)$ the distance from a fixed origin) is constant.
\end{quote}
\end{itemize}
\end{theorem}
\begin{proof}[{\bf Sketch of the proof:}] The key implications are $(2) \Rightarrow (1)$ and $(3) \Rightarrow (1)$. Both proceed by contradiction, so assume the existence of a  unit speed geodesic $\gamma$ defined on a maximal finite interval $[0,T)$, and pick a small compact set $K$ not intersecting $\gamma([0,T))$ (this is possible since $\gamma$ is diverging).\\ 
\noindent $(2) \Rightarrow (1)$. Apply the AK-duality to produce a Khas'minskii potential $w\in E(X \backslash K)$. By restriction, the function $u \doteq w \circ \gamma$ is $E$-subharmonic on $[0,T)$, that is, any $C^2$ test $\phi$ touching $u$ from above shall satisfy $|\phi'| \le 1$ at touching points. However, since $u\le 0$ and $T<+\infty$, we can choose a line with derivative strictly less than $-1$ lying above the graph of $u$: translating the line downwards up to the first touching point we get a contradiction. Thus, $T = +\infty$ and $X$ is complete.\\
\noindent $(3) \Rightarrow (1)$:
Pick an exhaustion of $X$ by smooth, relatively compact domains $\Omega_j$ with $K \Subset \Omega_1$. As we said before, we exploit the existence of a (unique) $\infty$-capacitor $u_j$ for $(K, \Omega_j)$ (see \cite{juutinen, crandall_visit}), that satisfies
\begin{equation}\label{equation}
\left\{ \begin{array}{l}
u_j \ \text{ is $F_\infty$-harmonic on $\Omega_j \backslash K$,} \\[0.2cm]
u_j = 1 \quad \text{on } \, \partial K, \quad u_j = 0 \quad \text{on } \, \partial \Omega_j.
\end{array}\right.
\end{equation}
By comparison (Theorem 2.27 in \cite{maripessoa}), and since $\{u_j\}$ is equi-Lipschitz because of the minimization properties of $u_j$, the sequence $v_j = 1-u_j$ subconverges locally uniformly to a $F_\infty$-harmonic, Lipschitz function $v_\infty \ge 0$. Applying the Ahlfors property on $X\backslash K$ we get that $v_\infty=0$ on $X \backslash K$. Now, setting $w_j = v_j \circ \gamma$, we have $w_j(0)=0$, $w_j =1$ after some $T_j<T$, and by integration, $1/T \le \|w_j'\|_\infty \le C$ on $[0,T)$, for each $j$. This is impossible, since $w_j \ra 0$ locally uniformly.
\end{proof}

\subsection{The Hessian principle and martingale completeness}

According to 3) in Example \ref{ex_importante}, we formally define the viscosity, weak and strong Hessian principles in terms of Ahlfors properties. Let us consider the subequations $F = \{\lambda_1(A) \ge -1\}$ and $E = \{\vert p\vert \leq 1\}$, whose duals are $\widetilde F = \{\lambda_m(A)\ge 1\}$ and $\widetilde E = \{\vert p\vert \ge 1\}$. Then, $X$ satisfies:
\begin{itemize}
\item[-]  the viscosity, weak Hessian principle if the Ahlfors property holds for $\widetilde F$;
\item[-]  the viscosity, strong Hessian principle if the Ahlfors property holds for $\widetilde F \cup \widetilde E$.
\end{itemize}
The $r$ independence on $F$ and $E$ in the above definition are just for convenience. The properties could be stated as in Example \ref{ex_importante} by making use of a pair of functions $(f,\xi)$ satisfying $(f\xi)$. As discussed in the introduction, there are evidences that an Hessian principle, either weak or strong, be related with the martingale completeness.  Perhaps surprisingly, exploiting the low regularity and the AK-duality, we found that the two Hessian principles are equivalent, and that the martingale completeness is necessary for the validity of them. Apart from a regularity issue, this answers a question (Question 70) raised in \cite{prs_overview} (see also \cite{prs_milan}).  

\begin{theorem}\label{teo_hessianmax}
Let $X$ be a Riemannian manifold. Then, the following properties are equivalent:
\begin{itemize}
\item[(1)] $X$ satisfies the viscosity, weak Hessian principle;
\item[(2)] $X$ satisfies the viscosity, strong Hessian principle;
\item[(3)] $F \cap E$ has the Khas'minskii property with $C^\infty$ potentials.
\end{itemize}
In particular, all the above assertions imply that $X$ is martingale (and so, geodesically) complete.
\end{theorem}
\begin{proof}[{\bf Idea of the proof:}]
As a consequence of AK-duality and the flexibility in the choice of $(f,\xi)$, each of $(1)$ and $(2)$ is equivalent to the corresponding Khas'minskii property for the dual subequation, that is, for $F$ and $F \cap E$, respectively. The key facts, very specific to such an $F$, are the following:
\begin{itemize}
\item by exploiting Greene-Wu's techniques in \cite{greene_wu}, we can approximate a Khas'minskii potential for $F$ with a smooth Khas'miskii potential, call it $w$. Up to playing with $(f,\xi)$ and extending $w$ on the entire $X$, we can assume that $w$ satisfies 
\begin{equation}\label{w_hessian}
w <0 \quad \text{on } \, X, \qquad w (x) \ra -\infty \quad \text{if $x$ diverges,} \qquad \nabla^2 w \ge - w \metric \quad \text{on } \, X.
\end{equation}
\item Integrating along the flow lines of $\nabla w$ and applying ODE comparison, it is possible to prove that $|\nabla w| \le w$ on $X$. Starting from $w$, it is therefore easy to construct a weak Khas'minskii potential for $F \cap E$ that is smooth. AK-duality again and Greene-Wu approximation yields the full Khas'minskii property.
\end{itemize}
As said above, by work of M. Emery \cite{emery} property $(3)$ is known to imply the martingale completeness of $X$. 
\end{proof}

\begin{remark}
\emph{In order to check the viscosity Hessian principle, we can only consider semiconcave\footnote{By definition, a function $u$ is semiconcave if and only if, locally, there exists $v \in C^2$, such that $u+v$ is concave when restricted to geodesics.} functions, which are locally Lipschitz and $2$-times differentiable a.e.. Thus, regarding to regularity, the viscosity Hessian principle is very close to the classical $C^2$ Hessian principle.
}
\end{remark}

\subsection{Laplacian principles}
Differently from the Hessian principle, in view of elliptic estimates for semilinear equations, the viscosity weak Laplacian principle is equivalent to its corresponding classical $C^2$ principle, that is, to the stochastic completeness of $X$. In this case, the AK-duality improves on the original results in \cite{khasminski, marivaltorta}. Regarding the viscosity, strong Laplacian principle, that is, the Ahlfors property for the subequation $\{\tr(A) \ge 1\}\cup \{\vert p\vert \ge 1\} = \widetilde F \cup \widetilde E$, its equivalence with the classical, $C^2$ one (that is, Yau's principle) seems quite delicate and is currently unknown. However, the AK-duality in Corollary \ref{cor_main} guarantees the following:

\begin{theorem}
Let $X$ be a Riemannian manifold. Then, the following statements are equivalent:
\begin{itemize}
\item[(1)] $X$ satisfies the viscosity, strong Laplacian principle;
\item[(2)] $F \cap E$ has the (weak) Khas'minskii property.
\end{itemize}
In particular, any manifold satisfying the viscosity, strong Laplacian principle must be (geodesically) complete.
\end{theorem}

\section{Partial Trace (Grassmannian) operators}

In the context of submanifolds it is interesting to consider extrinsic conditions instead of constrain directly the geometry of the submanifold. For instance, many applications of the Omori-Yau maximum principles (cf. \cite{aliasmastroliarigoli, aliasbessadajczer, bessalimapessoa}) have been investigated in that spirit. Specifically, when $\sigma : X^m \ra Y^n$ is an isometric immersion, and $F \subset J^2(Y)$ is a subequation, the pull-back $\sigma^*F$ induces a subset $H \doteq \overline{\sigma^*F}$, maybe only satisfying the conditions $(P)$ and $(N)$. In some relevant examples, like those in $(\EE 2), (\EE 3)$ and their complex analogues in $(\EE 5)$, the induced $H$ is nontrivial and the following question is therefore natural:
\begin{quote}
\emph{can we transplant the Ahlfors property from $\widetilde{F}$ on $Y$ to $\widetilde{H}$ on $X$?}
\end{quote}
Trying to address the problem by contradiction, that is, assuming that the Ahlfors property does not hold for $\widetilde{H}$, one would need to extend a nontrivial $\widetilde{H}$-subharmonic function on $X$ to the entire $Y$. This seems quite difficult to achieve, especially if $X$ is merely immersed. On the contrary, the use of AK-duality makes the problem feasible. In particular, we can obtain the following result:
 
\begin{theorem}\label{teo_immersions}
Let $\sigma : X^m \ra Y^n$ be a proper isometric immersion. Assume that either
\begin{itemize}
\item[i)] $F_f$ is the universal subequation in $(\EE 2)$, $(\EE 5)$ with $k \le m$, and $\sigma$ has bounded second fundamental form $\II$;\vspace{0.1cm}
\item[ii)] $F_f$ is the universal subequation in $(\EE 3)$ with $k \le m$, and
$$
\sup \Big\{ \big|\tr_{\mathcal{V}}\II(x)\big| \ \ : \ \  x \in X, \ \mathcal{V} \le T_xX \ \text{ $k$-dimensional}\Big\} < +\infty.
$$
\end{itemize}
Then,  
$$
\text{$\widetilde{F_f} \cup \widetilde{E_\xi}$ has the Ahlfors property on $Y$} \quad \Longrightarrow \quad \text{$\widetilde{F_f} \cup \widetilde{E_\xi}$ has the Ahlfors property on $X$},
$$
for some (any) pair $(f,\xi)$ satisfying \eqref{fxi}.
\end{theorem}
\begin{proof}[{\bf Idea of the proof:}] It is conceptually quite simple: in our assumptions, AK-duality holds for each of $(\EE 2),(\EE 3)$ and $(\EE 5)$, and therefore, $F_f \cap E_\xi$ has the Khas'minskii property for some (any) such $(f,\xi)$. Given an arbitrary potential $\bar w$ for $F_f \cap E_\xi$ on $Y$, by the flexibility in the choice of $(f,\xi)$ and the properness of $\sigma$, the composition $w \doteq \bar w \circ \sigma$ should correspond to a weak Khas'minskii potential for $F_g \cap E_\xi$ on $X$, where $g$ just depends on $(f,\xi)$. To check this claim, one uses the standard chain rule formula
\begin{equation}\label{chainrule}
\nabla^2 w(X,Y) = \bar \nabla^2 w(\sigma_* X, \sigma_* Y) + \langle \bar \nabla w, \II(X,Y) \rangle,
\end{equation}
where $\nabla, \bar \nabla$ are the connections on $X$ and $Y$, respectively. The adaptation to viscosity solutions, however, makes the proof of the claim subtler from the technical point of view. The arbitrariness of $\bar w$ and of these choices guarantees the validity of the weak Khas'minskii property on $X$. Then, AK-duality again implies the desired conclusion. 
\end{proof}

\begin{remark}
\emph{The presence of $\bar \nabla w$ in \eqref{chainrule} forces to include the eikonal in the Ahlfors properties, otherwise more restrictive assumptions have to be imposed on $X$. In fact, without a gradient bound, it is possible to control the last term in \eqref{chainrule} if and only if the second fundamental form $\II$ is trace-free on suitable subspaces $\cal V$. For instance, if $\sigma : X^m \ra Y^n$ is a proper minimal immersion and $F_f$ is the subequation described in $(\EE 2)$ with $k=m$, that is, $\widetilde{F_f} = \big\{ \lambda_{n-m+1} + \ldots + \lambda_n(A) \ge f(r)\big\}$, the validity of the Ahlfors property for $\widetilde F_f$ on $Y$ implies that $X$ is stochastically complete (i.e. $X$ has the viscosity, weak Laplacian principle).}
\end{remark}

\begin{remark}
\emph{A result similar to Theorem \ref{teo_immersions} can be stated for Riemannian submersions, cf. \cite[Thm. 7.8]{maripessoa}.
}
\end{remark}

The class of partial trace operators, example $(\EE 2)$, helps to understand the geometry of submanifolds. Thus, having in mind Theorem \ref{teo_immersions}, it is important to investigate sufficient geometric conditions that imply the validity of the Ahlfors property for this kind of operators. Inspired by the seminal papers of Omori \cite{omori} and Yau \cite{yau} (that correspond to cases $k=1$ and $k=m$ in $(\EE 2)$, respectively), we will focus on conditions involving the $k$-th Ricci curvature
\begin{definition}\label{def_medioricci}
Let $Y^n$ be an $n$-dimensional manifold, and let $k \in \{1,\ldots, n-1\}$. The $k$-th Ricci curvature is the function
$$
\begin{array}{ccccl}
\Ricc^{(k)} & : &  TY &  \longrightarrow  &  \R \\[0.3cm]
&& v & \longmapsto & \disp \inf_{\begin{array}{c}
\mathcal{W}_k \le v^\perp \\
\dim \mathcal{W}_k = k
\end{array}} \left( \frac{1}{k} \sum_{j=1}^k \mathrm{Sect}(v \wedge e_j)\right), 
\end{array}
$$
where $\{e_j\}$ is an orthonormal basis of $\mathcal W_k$.
\end{definition}
We recall that bounding from below the $k$-th Ricci curvature is an intermediate condition between the corresponding bounding for the sectional and Ricci curvature. In the next result, 
$$
F_f= \big\{ \lambda_1(A) + \ldots + \lambda_{k+1}(A) \ge f(r)\big\} 
$$
and the functions $f$ and $\xi$ satisfying $(f,\xi)$. Having fixed an origin $o$, we denote with $\rho(x)$ the distance from $o$ and with $\cut(o)$ the cut-locus of $o$, cf. \cite{docarmo}.

\begin{theorem}\label{prop_sufficientiFk}
Let $Y^n$ be complete, and assume that 
\begin{equation}\label{ine_riccimedio}
\Ricc^{(k)}_x(\nabla \rho) \ge - G^2\big(\rho(x)\big) \qquad \forall \, x \not \in \cut(o),
\end{equation}
for some $k \in \{1, \ldots, n-1\}$ and some $G$ satisfying
$$
0 < G \in C^1(\R^+_0), \qquad G' \ge 0, \qquad G^{-1} \not \in L^1(+\infty).
$$
Then, $Y$ has the Ahlfors property for $\widetilde{F_f} \cap \widetilde{E_\xi}$. Moreover, if $\sigma : X^m \ra Y^n$ is a proper isometric immersion, $k+1\le m \le n-1$ and the eigenvalues $\mu_1 \le \ldots \le \mu_m$ of the second fundamental form $\II$ satisfy 
\begin{equation}\label{ipo_tracepatrial_medioricci}
\max\Big\{ \big|\mu_1 + \ldots + \mu_{k+1}\big|, \big|\mu_{m-k} + \ldots + \mu_m\big|\Big\} \le C G(\rho \circ \sigma)\qquad \text{on } \, X,
\end{equation}
for some constant $C>0$, then the Ahlfors property for $\widetilde{F_f} \cap \widetilde{E_\xi}$ holds on $X$. In particular, if $m=k+1$ and the mean curvature satisfies 
$$
\big|H\big| \le C G(\rho \circ \sigma),
$$
then $X$ has the viscosity, strong Laplacian principle.
\end{theorem}

%

\section{AK-duality and polar sets}

With the aid of AK-duality, we can characterize polar (hence, removable) sets for subequations in terms of preservation of the Ahlfors property. The study of removable sets for linear and nonlinear equations is a classical subject with a long history, and the interested reader can consult the recent \cite{HL_removable} and the references therein for further insight. There, the problem is set in the language of subequations, and to introduce our application we first need to recall some terminology. We say that a subequation $F \subset J^2(X)$ is a 
\begin{itemize}
\item[-] truncated cone subequation if each fiber $F_x$ over a point $x \in X$ is a truncated cone, that is, it satisfies the following property:
$$
\text{if } \ \  J \in F_x, \ \ \text{ then } \ \ tJ \in F_x \quad \forall \, t \in [0,1]. 
$$
\item[-] convex cone subequation if each fiber $F_x$ over a point $x \in X$ is a convex cone.
\end{itemize}

A subequation $M$ is called a \emph{monotonicity cone} for $F$ if $M$ is a convex cone subequation and $F+M \subset F$, that is, $J_1+J_2 \in F$ whenever $J_1 \in F$ and $J_2 \in M$. In this case, we say that $F$ is $M$-monotone. By duality, also $\widetilde{F}$ is $M$-monotone, that is, 
$$
F + M \subset F \qquad \Longrightarrow \qquad \widetilde{F} + M \subset \widetilde{F}.
$$
In particular, since $M$ is a convex cone subequation, $M + M \subset M$ and thus $\widetilde{M}$ is $M$-monotone and $M \subset \widetilde{M}$. In general, $\widetilde{M}$ is a cone subequation much larger than $M$ and it is non-convex. Moreover, it is maximal among $M$-monotone cone subequations: indeed, if $F$ is a cone subequation that is $M$-monotone, then $0 \in F$ and thus $M = 0 + M \subset F$. Duality gives $\widetilde{F} \subset \widetilde{M}$.

\begin{example}
\emph{\begin{itemize}
\item[1)] On an almost complex, Hermitian manifold $X$, consider the subequations
$$
F_j = \big\{ \lambda_j(A^{(1,1)}) \ge 0\big\}, \qquad 1 \le j \le m,
$$
that are the branches of the complex Monge-Amp\`ere equation $\det(\nabla^2 u)^{(1,1)} = 0$. Then, $F_1$ (the only branch that is convex) is a monotonicity cone for each $F_j$.
\item[2)] Let $F_1,\ldots, F_k$ be the branches of the $k$-Hessian equation $\sigma_k(\nabla^2 u) = 0$. Then, the smallest branch $F_1$ is a monotonicity cone for $F_j$ for each $j$.
\item[3)] Let $F$ be a universal subequation, and let $M_k$ be the subequation in $(\EE 2)$ describing $k$-subharmonic functions:
$$
M_k = \big\{ \lambda_1(A) + \ldots + \lambda_k(A) \ge 0 \big\}.
$$
It is proved in \cite{HL_removable} that $M_k$ is a monotonicity cone for $F$ if and only if the \emph{Riesz characteristic} of $F$, $p_F$,  satisfies $p_F \ge k$. The Riesz characteristic of a universal subequation $F$ is an explicitably computable quantity defined as follows: 
$$
p_F = \sup \Big\{ t > 0 \ \ : \ \ I - t \Pi_{v} \in F \ \ \  \forall \, x \in X, \ v \in T_xX \ \text{with} \ |v|=1 \Big\}, 
$$
where $\Pi_v$ is the orthogonal projection onto the span of $v$. The interested reader can consult Section 11 of \cite{HL_removable} for further information.
\end{itemize}
} 
\end{example}

\begin{definition}
Let $F$ be a subequation. We say that a function $\psi \in USC(X)$ is \emph{polar for a set $\Sigma$} if $\Sigma \equiv \{ x : \psi(x) = -\infty\}$. A closed subset $\Sigma \subset X$ is called \emph{$F$-polar} if there exists an open  neighbourhood $\Omega \supset K$ and $\psi \in F(\Omega)$ polar for $\Sigma$. The set $\Sigma$ is called $C^2$ $F$-polar if, moreover, $\psi \in C^2(\Omega \backslash \Sigma)$.
\end{definition}

Sets that are $C^2$ $M$-polar are removable for subequations having $M$ as a monotonicity cone, cf. \cite[Thm. 6.1]{HL_removable}. The result is particularly effective when $F = \widetilde{M}$. Concerning Example 2) above, $M_k$-polar sets are very well understood in the Euclidean space (and, with some technical modifications, also on manifolds). In particular, if $\Sigma$ has locally finite $(p-2)$-dimensional Hausdorff measure for some $p< p_F$, then $\Sigma$ is $M_k$-polar (cf. \cite{HL_removable}, Theorems 11.4, 11.5 and 11.13). Moreover, if $\Sigma \subset \R^m$, having locally finite $(p_F-2)$-dimensional Hausdorff measure is enough to guarantee the $M_k$-polarity, cf. \cite[Thm. A]{HL_removable}.

In our setting, we consider subequations for which the AK-duality holds and thus we restrict to assume at least $(\HH 1)$. Consequently, since $F$ is a closed subset, the constant function $0$ is $F$-subharmonic. Any monotonicity cone $M$ for $F$ satisfies 
$$
M = \{0\} + M \subset F,
$$
thus any $M$-polar subset is automatically $F$-polar.

\begin{theorem}\label{teo_polar}
Let $F$ be an admissible subequation that is locally jet-equivalent to a universal subequation and satisfies $(\HH 1), (\HH 2)$. Suppose that $\widetilde{F}$ has the Ahlfors property on $X$, and let  $\Sigma \subset X$ be a compact subset. Then, the following holds:
\begin{itemize}
\item[(i)] if $F$ is a truncated cone subequation, then 
$$
\begin{array}{c}
\text{$\widetilde{F}$ has the Ahlfors} \\[0.2cm]
\text{property on } \, X \backslash \Sigma 
\end{array}
\qquad \Longleftrightarrow \qquad \text{$\Sigma$ is $F$-polar};
$$
\item[(ii)] if $M$ is a monotonicity cone for $F$ and $\Sigma$ is $M$-polar, then $\widetilde{F}$ has the Ahlfors property on $X \backslash \Sigma$.
\end{itemize}
\end{theorem}

\begin{proof}
We first prove $(i)$.\\
$(\Rightarrow)$ By AK-duality, any fixed pair $(K, h)$ with $K \Subset X \backslash \Sigma$ admits a Khas'minskii potential $\psi$. Since $\psi(x) \ra -\infty$ as $x \ra \Sigma$, extending $\psi$ on $X \backslash K$ by setting $\psi = -\infty$ on $\Sigma$ gives a USC and $F$-subharmonic function on $X \backslash K$. Hence, $\Sigma$ is $F$-polar.\\
$(\Leftarrow)$ By $F$-polarity, fix $\Omega \supset \Sigma$ open and $\psi \in F(\Omega)$ satisfying $\Sigma = \{ \psi = -\infty\}$. The upper semicontinuity of $\psi$ implies that $\psi(x) \ra -\infty$ as $x \ra \Sigma$. By AK-duality, we can consider a Khas'minskii potential $z$ for a pair $(\Omega',h)$ with $\Omega \Subset \Omega' \Subset X$. Then, for $\delta \in (0,1]$, the family of functions $\{\delta w\}$ with 
\begin{equation}\label{def_w_polar}
w(x) = \left\{ \begin{array}{ll}
\psi(x) & \quad \text{if } \, x \in \Omega, \\[0.2cm]
z(x) & \quad \text{if } \, x \in X \backslash \Omega',
\end{array}\right.
\end{equation}
realizes the weak Khas'minskii property on $X \backslash \Sigma$. By AK-duality, $\widetilde{F}$ has the Ahlfors property on $X \backslash \Sigma$, as claimed.\\
To show $(ii)$, fix $\Omega \supset \Sigma$ open and $\psi \in M(\Omega)$ that satisfies $\Sigma = \{ \psi = -\infty\}$. By $(\HH 1)$ and since $F$ is a closed subset, the constant $0 \in F(X)$. Therefore, being $M$ a monotonicity cone, $\delta \psi = 0 + \delta \psi \in F(\Omega)$ for each $\delta  \in (0,1]$. Defining $w$ as in \eqref{def_w_polar}, the family $\{\delta w\}$  give again the desired weak Khas'minskii potentials.
\end{proof}

\begin{acknowledgement}
This work was completed when the second author was visiting the Abdus Salam International Center for Theoretical Physics (ICTP), Italy. He is grateful for the warm hospitality and for financial support. The authors would also like to thank the organizing and local committees of the INdAM workshop ``Contemporary Research in elliptic PDEs and related topics" (Bari, May 30/31, 2017) for the friendly and pleasant environment. 
\end{acknowledgement}

\bibliographystyle{amsplain}

\begin{thebibliography}{99.}


\bibitem{ahlfors} L.V. Ahlfors, \emph{An extension of Schwarz's lemma.} Trans. Amer. Math. Soc. 43 (1938), 359-364.



\bibitem{ahlforssario} L.V. Ahlfors and L. Sario, \emph{Riemann surfaces}. Princeton mathematical series 26, Princeton Univ. Press, 1960.
%
\bibitem{aliasbessadajczer} L.J. Al\'ias, G.P. Bessa and M. Dajczer, \emph{The mean curvature of cylindrically bounded submanifolds.} Math. Ann. 345 (2009), no. 2, 367-376.

%
%
\bibitem{aliasmastroliarigoli} L.J. Al\'ias, P. Mastrolia, and M. Rigoli, \emph{Maximum principles and geometric applications.} Springer Monographs in Mathematics. Springer, Cham, 2016. xvii+570 pp.
%
\bibitem{aliasmirandarigoli} L.J. Al\'ias, J. Miranda, and M. Rigoli, \emph{A new open form of the weak maximum principle and geometric applications.} Comm. Anal. Geom. 24 (2016), no. 1, 1-43.
%
%
\bibitem{aronsson} G. Aronsson, \emph{Extension of functions satisfying Lipschitz conditions.} Ark. Mat. 6 (1967), 551-561.
%
%
%
\bibitem{bar} C. B\"ar and F. Pf\"affle, \emph{Wiener measures on Riemannian manifolds and the Feynman-Kac formula.} Mat. Contemp. 40 (2011), 37-90. 

\bibitem{bardidalio} M. Bardi and F. Da Lio, \emph{On the strong maximum principle for fully nonlinear degenerate elliptic equations.} Arch. Math. (Basel) 73 (1999) 276-285.
%
\bibitem{barlesbusca} G. Barles and J. Busca, \emph{Existence and comparison results for fully nonlinear degenerate elliptic equations without zeroth-order term.} Comm. Partial Differential Equations 26 (2001), no. 11-12, 2323-2337.  
%
%
%
\bibitem{bessalimapessoa} G. P. Bessa, B. P. Lima and L. F. Pessoa, \emph{Curvature estimates for properly immersed $\phi_{h}$-bounded submanifolds} Ann. Mat. Pura ed Appl. 194 (2015), no. 1, 109-130.
%
%
%
%
%

\bibitem{bmpr} B. Bianchini, L. Mari, P. Pucci and M. Rigoli, \emph{On the interplay among maximum principles, compact support principles and Keller-Osserman conditions on manifolds}. Submitted, available at arXiv:1801.02102.

\bibitem{borbely} A. Borbely, \emph{A remark on the Omori-Yau maximum principle.} Kuwait J. Sci. Engrg. 39 (2012), no. 2A, 45-56. 
%
\bibitem{borbely_counter} A. Borbely, \emph{Stochastic Completeness and the Omori-Yau Maximum Principle.} J. Geom. Anal. (2017). doi:10.1007/s12220-017-9802-7.
%
%
%
%
%
%
%
%
\bibitem{caffacabre} L. Caffarelli and X. Cabre, \emph{Fully Nonlinear Elliptic Equations.} Colloquium Publications vol. 43, Amer. Math. Soc., Providence, RI, 1995.
%
%
\bibitem{chengyau} S.Y. Cheng and S.T. Yau, \emph{Differential equations on Riemannian manifolds and their geometric applications.} Comm. Pure Appl. Math. 28 (1975), no. 3, 333-354.
%
\bibitem{chengyau_minkovski} S.Y. Cheng and S.T. Yau, \emph{Maximal space-like hypersurfaces in the Lorentz-Minkowski spaces.} Ann. of Math. (2) 104 (1976), no. 3, 407-419.
 
%
%
\bibitem{crandall_visit} M.G. Crandall, \emph{A Visit with the $\infty$-Laplace Equation.} Calculus of Variations and Nonlinear Partial Differential Equations. Lecture Notes in Mathematics, vol. 1927, pp. 75–122. Springer, Berlin (2008).
%
\bibitem{CEG} M.G. Crandall, L.C. Evans and R.F. Gariepy, \emph{Optimal Lipschitz extensions and the infinity Laplacian.} Calc. Var. P.D.E. 13 (2001), no. 2, 123-139.
%
\bibitem{cil} M.G. Crandall, H. Ishii and P.L. Lions, \emph{User's guide to viscosity solutions of second-order partial differential equations.} Bull. Am. Math. Soc. 27 (1992), 1-67
%
%
%
%
%

\bibitem{docarmo} M.P. Do Carmo, \emph{Riemannian Geometry.} Mathematics: Theory and Applications, Birk\"auser Boston INC, Boston, MA, 1992.
\bibitem{dodziuk} J. Dodziuk, \emph{Maximum principle for parabolic inequalities and the heat flow on open manifolds.} Indiana Univ. Math. J. 32 (1983), no. 5, 703-716. 

\bibitem{DE} D.M. Duc and J. Eells, \emph{Regularity of exponentially harmonic functions.} Internat. J. Math. 2 (1991), 395-408.
%

\bibitem{emery} M. Emery, \emph{Stochastic Calculus in Manifolds.} Universitext. Springer-Verlag, Berlin, 1989.
%
\bibitem{ekeland_2} I. Ekeland, \emph{On the variational principle.} J. Math. Anal. Appl. 47 (1974), 324-353.
%
%
%
%
%
%
%
%
%
%
%
\bibitem{fontenelexavier} F. Fontenele and F. Xavier, \emph{Good shadows, dynamics and convex hulls of complete submanifolds.} Asian J. Math. 15 (2011), no. 1, 9-31.
 
%
\bibitem{G} L. G\"arding, \emph{An inequality for hyperbolic polynomials.} J. Math. Mech. 8 (1959), 957-965.
%
%
%
\bibitem{greene_wu} R.E. Greene and H. Wu, \emph{$C^\infty$ approximation of convex, subharmonic, and plurisubharmonic functions.} Ann. scient. Ec. Norm. Sup. $4^{\mathrm{e}}$ serie t.12, 47-84 (1979).
%
\bibitem{grigoryan} A. Grigor'yan, \emph{Analytic and geometric background of recurrence and non-explosion of the Brownian motion on Riemannian manifolds.} Bull. Amer. Math. Soc. 36 (1999), 135-249.
%
\bibitem{HL_primo} F.R. Harvey and H.B. Lawson Jr., \emph{Dirichlet duality and the non-linear Dirichlet problem.} Comm. Pure Appl. Math. 62 (2009), 396-443.
%
\bibitem{HL_dir} F.R. Harvey and H.B. Lawson Jr., \emph{Dirichlet Duality and the Nonlinear Dirichlet Problem on Riemannian Manifolds.} J. Differential Geom. 88 (2011), 395-482.
%
\bibitem{HL_plurisub}  F.R. Harvey and H.B. Lawson Jr., \emph{Geometric plurisubharmonicity and convexity: an introduction.} Adv. Math. 230 (2012), no. 4-6, 2428-2456. 
%
\bibitem{HL_existence} F.R. Harvey and H.B. Lawson Jr., \emph{Existence, uniqueness and removable singularities for nonlinear partial differential equations in geometry.} pp. 102-156 in ``Surveys in Differential Geometry 2013'', vol. 18, H. D. Cao and S. T. Yau eds., International Press, Somerville, MA, 2013.
%
\bibitem{HL_SMP} F.R. Harvey and H.B. Lawson Jr., \emph{Characterizing the strong maximum principle for constant coefficient subequations.} Rend. Mat. Appl. (7) 37 (2016), no. 1-2, 63-104.

\bibitem{HL_removable} F.R. Harvey and H.B. Lawson Jr., \emph{Removable singularities for nonlinear subequations.} Indiana Univ. Math. J. 63 (2014), no. 5, 1525-1552.
 
\bibitem{HL_garding} F.R. Harvey and H.B. Lawson Jr., \emph{G\"arding's theory of hyperbolic polynomials.} Comm. Pure Appl. Math. 66 (2013), no. 7, 1102-1128. 
%
%
%
%
%
%
%
%
\bibitem{holopainen} I. Holopainen, \emph{Nonlinear potential theory and quasiregular mappings on Riemannian manifolds.} Ann. Acad. Sci. Fenn. Ser. A I Math. Dissertationes 74 (1990), 45 pp.
%
%
%
%
\bibitem{imperapigolasetti} D. Impera, S. Pigola and A.G. Setti, \emph{Potential theory for manifolds with boundary and applications to controlled mean curvature graphs.} To appear on J. Reine Angew. Math.
%
%
%
\bibitem{jensen} R. Jensen, \emph{Uniqueness of Lipschitz extensions: minimizing the sup norm of the gradient.} Arch. Rational Mech. Anal. 123 (1993), no. 1, 51-74.
%
%
\bibitem{juutinen} P. Juutinen, \emph{Absolutely minimizing Lipschitz extensions on a metric space.} Ann. Acad. Sci. Fenn. Math. 27 (2002), no. 1, 57-67. 
%
%
\bibitem{karp} L. Karp, \emph{Differential inequalities on complete Riemannian manifolds and applications.} Math. Ann. 272 (1985), no. 4, 449-459.
%
%

\bibitem{kawohlkutev_SMP} B. Kawohl and N. Kutev, \emph{Strong maximum principle for semicontinuous viscosity solutions of nonlinear partial differential equations.} Arch. Math. (Basel) 70 (1998), no. 6, 470-478.

%
\bibitem{khasminski} R.Z. Khas'minskii, \emph{Ergodic properties of recurrent diffusion processes and stabilization of the solution of the {C}auchy problem for parabolic equations.} Teor. Verojatnost. i Primenen., Akademija Nauk SSSR. Teorija Verojatnoste\u\i\ i ee Primenenija 5 (1960), 196-214.
%
%
\bibitem{krylov} N.V. Krylov, \emph{On the general notion of fully nonlinear second-order elliptic equations.} Trans. Amer. Math. Soc. 347 (1995), no. 3, 857-895. 
%
\bibitem{kuramochi} Z. Kuramochi, \emph{Mass distribution on the ideal boundaries of abstract Riemann surfaces I.} Osaka Math. J. 8 (1956), 119-137. 

\bibitem{cokumero} P. Collin, R. Kusner, W.H. Meeks III and H. Rosenberg, \emph{The topology, geometry and conformal structure of properly embedded minimal surfaces.} J. Differential Geom. 67 (2004), no. 2, 377-393.
 
%
%
\bibitem{li} P. Li, \emph{Harmonic functions and applications to complete manifolds.} XIV Escola de Geometria Diferencial: Em homenagem a Shiing-Shen Chern. (2006).


%
%
%
\bibitem{litam} P. Li and L.-F. Tam, \emph{Harmonic functions and the structure of complete manifolds.} J. Diff. Geom. 35 (1992), 359-383.

\bibitem{liwang} P. Li and J. Wang, \emph{Complete manifolds with positive spectrum, II}. J. Diff. Geom. 62 (2002), 143-162.

\bibitem{luoeberhardt} Y. Luo and A. Eberhard, \emph{Comparison principle for viscosity solutions of elliptic equations via fuzzy sum rule.} J. Math. Anal. Appl. 307 (2005), 736-752.

%
%
\bibitem{maripessoa} L. Mari and L.F. Pessoa, \emph{Duality between Ahlfors-Liouville and Khas'minskii properties for non-linear equations.} To appear on Comm. Anal. Geom.

\bibitem{mazet} L. Mazet, \emph{A general halfspace theorem for constant mean curvature surfaces.} Amer. J. Math. 135 (2013), no. 3, 801-834.

%
\bibitem{mmr_berestycki} M. Magliaro, L. Mari and M. Rigoli, \emph{On a paper of Berestycki-Hamel-Rossi and its relations to the weak maximum principle at infinity, with applications.} Rev. Mat. Iberoam. (2017), in press.
%

\bibitem{maririgoli} L. Mari and M. Rigoli. \emph{Maps from Riemannian manifolds into non-degenerate Euclidean cones.} Rev. Mat. Iberoam. 26 (2010), no.3, 1057-1074.

%
\bibitem{marivaltorta} L. Mari and D. Valtorta, \emph{On the equivalence of stochastic completeness, Liouville and Khas'minskii condition in linear and nonlinear setting.} Trans. Amer. Math. Soc. 365 (2013), no. 9, 4699-4727.
%
%
\bibitem{meeksrosenberg} W.H. Meeks III and H. Rosenberg, \emph{Maximum principles at infinity.} J. Differential Geom. 79 (2008), no. 1, 141-165.

%
%
%
\bibitem{nakai} M. Nakai, \emph{On Evans potential.} Proc. Japan Acad. 38 (1962), 624-629.
%
\bibitem{sarionakai} M. Nakai and L. Sario, \emph{Classification theory of Riemann surfaces.} Springer, Berlin, 1970.

%
\bibitem{omori} H. Omori, \emph{Isometric immersions of Riemannian manifolds.} J. Math. Soc. Japan 19 (1967), 205-214.
%
%
%
%
%
%
%
%
\bibitem{prs_proceeding} S. Pigola, M. Rigoli and A.G. Setti, \emph{A remark on the maximum principle and stochastic completeness.} Proc. Amer. Math. Soc. 131 (2003), no. 4, 1283-1288.
%
%
\bibitem{prsmemoirs} S. Pigola, M. Rigoli and A.G. Setti, \emph{Maximum principles on Riemannian manifolds and applications.} Mem. Amer. Math. Soc. 174 (2005), no. 822.
%
\bibitem{prsnonlinear} S. Pigola, M. Rigoli and A.G. Setti, \emph{Some non-linear function theoretic properties of Riemannian manifolds.} Rev. Mat. Iberoam. 22 (2006), no. 3, 801-831.
%
\bibitem{prs_overview} S. Pigola, M. Rigoli and A.G. Setti, \emph{Maximum principles at infinity on Riemannian manifolds: an overview.} Workshop on Differential Geometry (Portuguese), Mat. Contemp. 31 (2006), 81-128.
%
\bibitem{prs} S. Pigola, M. Rigoli and A.G. Setti, \emph{Vanishing and finiteness results in Geometric Analysis. A generalization of the B\"ochner technique.} Progress in Math. 266, Birk\"auser, 2008.
%
\bibitem{prs_milan} S. Pigola, M. Rigoli, A.G. Setti, \emph{Aspects of potential theory on manifolds, linear and non-linear.} Milan J. Math. 76 (2008), 229-256.
%
\bibitem{pigolasetti_ensaio} S. Pigola and A.G. Setti, \emph{Global divergence theorems in nonlinear PDEs and geometry.} Ensaios Matem\'aticos [Mathematical Surveys], 26. Sociedade Brasileira de Matem\'atica, Rio de Janeiro, 2014. ii+77 pp.
%
%
%
%
%
\bibitem{PuRS} P. Pucci, M. Rigoli and J. Serrin, \emph{Qualitative properties for solutions of singular elliptic inequalities on complete manifolds.} J. Differential Equations 234 (2007), no. 2 507-543.
%
%
\bibitem{pucciserrin} P. Pucci and J. Serrin, \emph{The maximum principle.} Progress in Nonlinear Differential Equations and their Applications, 73, Birkh\"auser Verlag, Basel, 2007, x+235.
%
%
%
%
%
%
%
%
%
%
%
%
%
\bibitem{sha} J.P. Sha, \emph{p-convex riemannian manifolds.} Invent. Math. 83 (1986), 437-447.
%
\bibitem{stroockvaradhan} D.W. Stroock, S.R.S. Varadhan, \emph{Multidimensional diffusion processes.} Reprint of the 1997 edition. Classics in Mathematics. Springer-Verlag, Berlin, 2006. xii+338 pp. 
%
\bibitem{sungtamwang} C.-J. Sung, L.-F. Tam, J. Wang, \emph{Spaces of Harmonic Functions}, J. London Math. Soc. 61 (2000), no. 3, 789-806.

\bibitem{sullivan} F. Sullivan, \emph{A characterization of complete metric spaces.} Proc. Amer. Math. Soc. 83 (1981), no. 2, 345-346.

%
%
%
%
\bibitem{troyanov2} M. Troyanov, \emph{Parabolicity of manifolds.} Siberian Adv. Math. 9 (1999), no. 4, 125-150.
%
%
\bibitem{valto_reverse} D. Valtorta, \emph{Reverse Khas'minskii condition.} Math. Z. 270 (2011), no. 1, 165-177.

\bibitem{valtorta} D. Valtorta, \emph{Potenziali di Evans su variet\'a paraboliche.} Available at arXiv:1101.2618.


%
%
\bibitem{weston} J.D. Weston, \emph{A characterization of metric completeness.} Proc. Amer. Math. Soc. 64 (1977), no. 1, 186-188.
%
\bibitem{wu} H. Wu, \emph{Manifolds of partially positive curvature.} Indiana Univ. Math. J. 36  (1987), no. 3, 525-548.
\bibitem{yau} S.T. Yau, \emph{Harmonic functions on complete Riemannian manifolds.} Comm. Pure Appl. Math. 28 (1975), 201-228.

\bibitem{yau_schwarz} S.T. Yau, \emph{A general Schwarz lemma for K\"ahler manifolds.} Amer. J. Math. 100 (1978), no. 1, 197-203. 


%
%
%
\end{thebibliography}

\end{document}